\documentclass[a4paper,11pt,onecolumn]{amsart}
\usepackage{amsfonts}

\usepackage{amscd}
\usepackage{amsmath}
\usepackage{amssymb}
\usepackage{amsthm}
\usepackage{epsf}
\usepackage{latexsym}
\usepackage{verbatim}
\usepackage{color}
\usepackage{todonotes}
\usepackage{enumitem}
\input epsf.tex
\usepackage{bbm}
\usepackage{mathtools}

\usepackage{epstopdf}
\usepackage{graphicx}

\usepackage{centernot}
\usepackage{mathtools}
\usepackage{ stmaryrd }
\usepackage[all]{xy}

\usepackage{marginnote}

\theoremstyle{plain}
\newtheorem{theorem}{Theorem}
\newtheorem{definition}{Definition}[section]
\newtheorem{lemma}[definition]{Lemma}
\newtheorem{corollary}[definition]{Corollary}
\newtheorem{proposition}[definition]{Proposition}
\newtheorem*{theoremBHAispolycpx}{Theorem~\ref{thm:BHAispolycpx}}
\newtheorem*{theoremnminus1skel}{Theorem~\ref{thm:nminus1skel}}
\newtheorem*{theoremtransversalgeneric}{Theorem~\ref{thm:transversalgeneric}}
\newtheorem*{theoremJohnsonResult}{Theorem~\ref{thm:JohnsonResult}}
\newtheorem*{theoremoneboundedcomponent}{Theorem~\ref{thm:oneboundedcomponent}}
\newtheorem*{theoremParametrizedPiecewiseSmooth}{Theorem~\ref{thm:parametrizedPiecewiseSmooth}}

\theoremstyle{definition}

\theoremstyle{remark}
\newtheorem{remark}[definition]{Remark}

\newcommand{\Relu}{\textrm{ReLU}}

\newcommand{\CA}{\mathcal{C}(\mathcal{A})}

\newcommand{\Ci}{\mathcal{C}(F_i \circ \ldots \circ F_1)}
\newcommand{\Cii}{\mathcal{C}(F_{i-1} \circ \ldots \circ F_1)}
\newcommand{\Cm}{\mathcal{C}(F_{m} \circ \ldots \circ F_1)}
\newcommand{\Cmi}{\mathcal{C}(F_{m-1} \circ \ldots \circ F_1)}
\newcommand{\arch}{(n_0,\dots,n_m;1)}

\begin{document} 
\bibliographystyle{plain}
\title[Topological expressiveness of ReLU neural networks]{On transversality of bent hyperplane arrangements and the topological expressiveness of ReLU neural networks}
\author{J. Elisenda Grigsby}
\thanks{JEG was partially supported by Simons Collaboration grant 635578.}
\address{Boston College; Department of Mathematics; 522 Maloney Hall; Chestnut Hill, MA 02467}
\email{grigsbyj@bc.edu}

\author{Kathryn Lindsey}
\thanks{KL was partially supported by NSF grant number DMS-1901247.}
\address{Boston College; Department of Mathematics; 567 Maloney Hall; Chestnut Hill, MA 02467}
\email{lindseka@bc.edu}

\begin{abstract}
Let $F: \mathbb{R}^n \rightarrow \mathbb{R}$ be a feedforward ReLU neural network. It is well-known that for any choice of parameters, $F$ is continuous and piecewise affine-linear. We lay some foundations for a systematic investigation of how the {\em architecture} of $F$ impacts the geometry and topology of its possible decision regions, $F^{-1}((-\infty, t))$ and $F^{-1}((t, \infty))$, for binary classification tasks. Following the classical progression for smooth functions in differential topology, we first define the notion of a {\em generic, transversal} ReLU neural network and show that almost all ReLU networks are generic and transversal. We then define a partially-oriented linear $1$--complex in the domain of $F$ and identify properties of this complex that yield an obstruction to the existence of bounded connected components of a decision region. We use this obstruction to prove that a decision region of a generic, transversal ReLU network $F: \mathbb{R}^n \rightarrow \mathbb{R}$ with a single hidden layer of dimension $n+1$ can have no more than one bounded connected component.
\end{abstract}

\maketitle

\section{Introduction}
Neural networks have rapidly become one of the most widely-used tools in the machine learning toolkit. Unfortunately, despite--or, perhaps, because of--their spectacular success in applications, significant foundational questions remain. Of these, we believe many would benefit greatly from the direct attention of theoretical mathematicians, particularly those in the geometric topology, nonlinear algebra, and dynamics communities. An important goal of this paper and its sequels is to advertise some of these problems to those communities.

Recall that one can view a (trained) feedforward neural network as a particular type of function, $F: \mathbb{R}^n \rightarrow \mathbb{R}^m,$ between Euclidean spaces. The inputs to the function are data feature vectors and the outputs are often used to answer $m$--class classification problems by partitioning the input space into {\em decision regions} according to which component of the function output is maximized at that point. 

The main purpose of the present work is to present a framework for studying the question: {\em How does the architecture of a feedforward neural network constrain the topology of its decision regions?} Here, the {\em architecture} of a feedforward neural network refers simply to the dimensions of the hidden layers. The neural networks we consider here will be fully-connected ReLU networks without skip connections. The {\em topological expressiveness} of an architecture is the collection of possible homeomorphism types of decision regions that can appear as the parameters vary (cf. \cite{BianchiniScarselli, GussSalakhutdinov}).

First: why should the machine learning community care about {\em topological expressiveness}?

Recall that a cornerstone theoretical result in the study of neural networks--for a variety of activation functions, including the widely-used ReLU function that is our focus here--is the Universal Approximation Theorem (\cite{Cybenko, Hornik, AroraBasu}), which says that a sufficiently high-dimensional neural network can approximate any continuous function on a compact set to arbitrary accuracy. This is the version of  {\em representational power} or {\em expressiveness} frequently cited by practitioners as a guarantee that feedforward neural networks can yield a solution to any data question one might throw at them. 

Yet continuous functions can be quite poorly behaved, and  certain classes of poorly behaved continuous functions are undesirable targets for learning. For example, functions with high Lipschitz constants and ones whose partial derivatives are highly variable with respect to the input direction lead to the easy production of adversarial examples and hence to potentially poor generalization to unseen data (cf. \cite{Cisse, Tsuzuku, Madry}).  

It is important to remark at this point that homeomorphism is a very coarse equivalence relation. Two different decision regions can be homeomorphic and still have quite different geometric properties (shape, volume, etc.). However, homeomorphism is a good equivalence relation to consider on a first pass, because the most fundamental global features of the data are preserved by homeomorphism (number of connected components and higher homology groups, e.g.). The flip side of this observation is that if a particular architecture lacks the topological expressiveness to capture obvious features inherent in a well-sampled labeled data set, it is unlikely to generalize well to unseen data.

In a more practical direction, we believe topological complexity is the right lens through which to study the implicit bias of stochastic gradient descent, i.e., the behavior of neural network functions during training. One mystery is the unreasonable success of overparameterized networks--those for which the number of trainable parameters far exceeds the size of the data set--in learning functions with low training loss and good generalization to unseen data \cite{Belkin}. The classical bias-variance trade-off curve suggests that increasing the number of parameters should lead the model to learn high-complexity functions that overfit the training set; overparameterized networks defy this intuition. An examination of the relationship between the topological features of an architecture's loss landscape \cite{Nguyenetal} and of its realizable functions \cite{RolnickKording} may shed light on this phenomenon. We lay groundwork for this in \cite{GLMW}.

Last, we remark that ReLU and other piecewise-linear (PL) activation functions are now the default choice in the community for regression problems, which in turn underlie almost every other machine learning problem. Unfortunately, the resulting PL neural network functions are much trickier to study than smooth ones, because the classical foundational work in differential topology takes considerable effort to translate to the PL setting. ReLU, in particular, often gives rise to quite degenerate functions. In the sequel, we argue that this is a feature and not a bug \cite{GLM}. 

We focus here on the simplest case of feedforward neural networks, $F: \mathbb{R}^n \rightarrow \mathbb{R}$, with $1$--dimensional output. Such a network is often used to answer {\em binary} (aka Yes/No) classification problems by choosing a threshold, $t \in \mathbb{R}$, and declaring the sublevel set of $t$ to be the ``N"  decision region, the superlevel set of $t$ to be the ``Y" decision region, and the level set of $t$ to be the decision boundary. 
That is: 
\begin{align} \label{decisionregionsnotation}
N_F(t) &:= F^{-1}((-\infty, t)) \nonumber \\
B_F(t) &:= F^{-1}(\{t\})\\
Y_F(t) &:= F^{-1}((t, \infty)). \nonumber
\end{align}

Classical results in differential topology now tell us that if $F$ were smooth, we could perturb $F$ slightly to be Morse, and the indices and values of its critical points would then provide strong information about the topology of its decision regions. Although a ReLU neural network map $F$ is not smooth, $F$ will {\em typically}\footnote{We'd like to use the word {\em generically} here, but the term {\em generic} is unavoidably used in a different context later in the paper (Definition \ref{defn:genericlayer}).} (i.e., for almost all choices of parameters) be sufficiently well-behaved that the information needed to understand much of the topology of its decision regions is extractable directly from the neural network parameters. 

We begin by reviewing some standard results in the theory of affine hyperplane arrangements and convex polyhedra, relying heavily on Grunert's work \cite{Grunert} on polyhedral complexes 
and Hanin-Rolnick's work \cite{HaninRolnick} (see also \cite{Montufar, Serra, HaninRolnick2}) generalizing the classical notion of {\em hyperplane arrangements} to so-called {\em bent hyperplane arrangements} (Definition \ref{defn:benthyperarr}).  
Note that the appearance of polyhedral complexes in the study of ReLU networks is well-known, and made explicit, e.g., in the relationship between ReLU neural networks with rational parameters and tropical rational functions \cite{Zhang2018TropicalGO}. Our contribution here is to formalize the relationship and open a path for applying classical ideas in differential topology to study decision regions. We summarize our main results below. More precise versions and proofs of these theorems appear in later sections as indicated.

\begin{theorem} \label{thm:BHAispolycpx} Let $F: \mathbb{R}^n \rightarrow \mathbb{R}$ be a ReLU neural network map. $F$ is continuous and affine-linear on the cells of a canonical realization of the domain, $\mathbb{R}^{n}$, as a polyhedral complex, $\mathcal{C}(F)$. 
\end{theorem}

Moreover, when $F$ is  {\em transversal} (Definition \ref{defn:GenericBHA}),  we explicitly identify cells of this polyhedral complex with natural objects defined by Hanin-Rolnick:

\begin{theorem} \label{thm:nminus1skel} Let $F: \mathbb{R}^n \rightarrow \mathbb{R}$ be a transversal ReLU neural network map. The $n$--cells of the canonical polyhedral complex, $\mathcal{C}(F)$, are the closures of the {\em activation regions} (Definition \ref{defn:activationregion}) of $F$, and the $(n-1)$--skeleton of $\mathcal{C}(F)$ is the {\em bent hyperplane arrangement} (Definition \ref{defn:benthyperarr}) of $F$.
\end{theorem}

We slightly extend classical transversality results (Theorems \ref{thm:Transmaps} and \ref{thm:ParTrans}) and incorporate the notion of a {\em generic} (Definition \ref{defn:genericlayer}) neural network  to obtain:

\begin{theorem} \label{thm:transversalgeneric} Almost all ReLU neural networks are generic and transversal.
\end{theorem}

Letting $D$ denote the dimension of the parameter space of a neural network, we show that a parametrized family of neural networks (Definition \ref{ParameterizedNN}) $\mathcal{F}: \mathbb{R}^{n_0} \times \mathbb{R}^{D} \rightarrow \mathbb{R}$ is \emph{piecewise smooth}  in the following sense: 

\begin{theorem} \label{thm:parametrizedPiecewiseSmooth}
Every parametrized family of ReLU neural networks  $\mathcal{F}$ is smooth on the complement of a codimension 1 algebraic set.  
\end{theorem}

Many of the key observations in Theorems \ref{thm:BHAispolycpx}--\ref{thm:parametrizedPiecewiseSmooth} were proved in \cite{HaninRolnick}.  The theorems above place those results in a broader context. Once we've established 
these foundational results, 
 we turn our attention to addressing some first questions about architecture's impact on topological expressiveness. We begin by using the framework developed above to recast and reprove the result of Beise-Cruz-Schr{\"o}der \cite{DDCruz}, proved independently by Johnson \cite{Johnson} and Hanin-Sellke \cite{HaninSellke}, that inspired this study:\footnote{The arguments of Beise-Cruz-Schr{\"o}der and Johnson hold for a large class of activation functions, including the ReLU activation function studied here. Hanin-Sellke prove the statement only for ReLU, but their paper has a broader scope.}

\begin{theorem} \label{thm:JohnsonResult} For any integer $n \geq 2$,  let $F: \mathbb{R}^n \rightarrow \mathbb{R}$ be a ReLU neural network, all of whose hidden layers have dimension $\leq n$.  Then for any decision threshold $t \in \mathbb{R}$, each of $Y_F(t)$ $B_F(t)$, and $N_F(t)$ is either empty or unbounded. 
\end{theorem}

We also have the following new application. See \cite{Nguyenetal} for related results:

\begin{theorem} \label{thm:oneboundedcomponent} Let $F: \mathbb{R}^n \rightarrow \mathbb{R}^{n+1} \rightarrow \mathbb{R}$ be a ReLU neural network with input dimension $n$ and a single hidden layer of dimension $n+1$. Each decision region of $F$ associated to a transversal threshold can have no more than one bounded connected component.
\end{theorem}

A crucial player in the proof of Theorem \ref{thm:oneboundedcomponent} is the $1$--skeleton, $\mathcal{C}(F)_1$, of the polyhedral complex, $\mathcal{C}(F)$, which is naturally endowed with a partial orientation pointing in the direction in which $F$ increases (Definition \ref{defn:oriented1skel}). 

Note that the partial orientation data can be extracted directly from the weight matrices of the neural network using the chain rule (Lemma \ref{lem:compgradFregion}). 

This paper is heavy on definitions and notation, since we pulled from a variety of sources to lay necessary foundations for a consistent and general theory. Some sections may therefore be safely skimmed on a first reading and referenced only as needed to understand the proofs of the main results. Sections \ref{ss:hyperplanearrangements} and \ref{sec:Binary} largely fall into this category. Similarly, in Section \ref{sec:PWSmooth} we establish important results about parameterized neural network maps and prove Theorem \ref{thm:parametrizedPiecewiseSmooth}, but nothing in this section is referenced elsewhere in the paper.

Sections \ref{sec:PolyCpx} and \ref{sec:Trans} establish notation and key terminology.  We do the bulk of the technical work in Section \ref{sec:PolyMapsTransThresh}, where we establish necessary transversality results; Section \ref{sec:BHA}, where we prove Theorem \ref{thm:BHAispolycpx}; and Section \ref{sec:TransNN}, where we prove Theorems \ref{thm:nminus1skel} and \ref{thm:transversalgeneric}. The new applications can be found in Section \ref{sec:TopApp}, where we reprove Theorem \ref{thm:JohnsonResult} and prove Theorem \ref{thm:oneboundedcomponent}.
 
\subsection*{Acknowledgements}
The authors would like to thank Jesse Johnson for proposing many of the questions at the heart of this investigation; Boris Hanin for illuminating conversations during the {\em Foundations of Deep Learning} program at the Simons Institute for the Theory of Computing in the summer of 2019 and numerous follow-on discussions that also included David Rolnick and Atlas Wang; Jenna Rajchgot for very helpful discussions about algebraic geometry; Yaim Cooper and Jordan Ellenberg for inspiration and encouragement. We would also like to thank the anonymous reviewers for pointing out gaps (now fixed) in some arguments and for a number of excellent suggestions that greatly improved the exposition.

\section{Layer maps and hyperplane arrangements} \label{ss:hyperplanearrangements}

In what follows, let 
\begin{itemize}
	\item $\Relu:\mathbb{R} \to \mathbb{R}$ denote the function $\Relu(x) := \max\{0,x\}$, and
	\item $\sigma:\mathbb{R}^n \to \mathbb{R}^n$ denote the function that applies $\Relu$ to each coordinate.
\end{itemize}
 
  \begin{definition}\label{d:neuralnetwork} Let $n_0 \in \mathbb{N}$.  A \emph{neural network defined on $\mathbb{R}^{n_0}$ with ReLU activation function on all hidden layers and one-dimensional output} is a finite sequence of natural numbers $n_1,\dots,n_m$ together with affine maps $A_i:\mathbb{R}^{n_{i-1}} \to \mathbb{R}^{n_{i}}$ for $i = 1,\dots,m+1$, where $n_{m+1}=1$.   This determines a function $F:\mathbb{R}^{n_0} \rightarrow \mathbb{R}$,  which we call the associated \emph{neural network map}, given by the composition
  \[\mathbb{R}^{n_0}  \xrightarrow{F_1 = \sigma \circ A_1} \mathbb{R}^{n_1}  \xrightarrow{F_2 = \sigma \circ A_2} \mathbb{R}^{n_2}  \xrightarrow{F_3 = \sigma \circ A_3}\dots \xrightarrow{F_m = \sigma \circ A_{m}} \mathbb{R}^{n_m}  \xrightarrow{G = A_{m+1}} \mathbb{R}^1.\]
   Such a neural network is said to  be of
 \emph{architecture} $\arch$, 
 \emph{depth}  $m+1$, and \emph{width} 
 $\max \{n_1,\ldots,n_m,1\}$. 
  The $k^{\textrm{th}}$ \emph{layer map} of such a neural network is the composition $F_k = \sigma \circ A_k$ for $k = 1,\dots,m$ and is the map $G = A_k$ for $k=m+1$. 
\end{definition}

\begin{remark}
Note that in Definition \ref{d:neuralnetwork} the activation function on the final layer map is the Identity function, not $\sigma$. Accordingly, we use the notation $G$ on the output layer map to distinguish it from the hidden layer maps, $F_k$.
\end{remark}

An affine map $A:\mathbb{R}^n \to \mathbb{R}^m$ is specified by a weight matrix $W \in M_{m \times n}(\mathbb{R})$ and a bias vector $\vec{b} \in \mathbb{R}^m$, as follows. Let $(W|\vec{b})$ denote the $m \times (n+1)$ matrix whose final column is $\vec{b}$. For each $\vec{x} \in \mathbb{R}^n$, let $\vec{x}':= (\vec{x},1)$ be the element of $\mathbb{R}^{n+1}$ whose first $n$ coordinates are those of $x$ and whose last coordinate is $1$.  Thus $A(\vec{x})= (W|\vec{b})\vec{x}'$.

For a network architecture  $\arch$, we will denote by \[D := \sum_{i=0}^m (n_i+1)n_{i+1}\] the \emph{total dimension} of the parameter space of neural networks of that architecture. When unclear from context, we will specify the architecture in the notation for the dimension: $D\arch$.

\begin{definition} \label{ParameterizedNN} Let $(n_0, \ldots, n_m ; 1)$ be a network architecture. The {\em parameterized family of ReLU neural networks with architecture $\arch$} is the map
 $\mathcal{F} : \mathbb{R}^{n_0} \times \mathbb{R}^{D} \rightarrow \mathbb{R}$
 defined as follows. For each ${\bf s} \in \mathbb{R}^{D}$, 
 $\mathcal{F}_{\bf s}: \mathbb{R}^{n_0} \times \{{\bf s}\} \rightarrow \mathbb{R}$ is the ReLU neural network map  associated to the  weights and biases given by ${\bf s}$.
\end{definition}

Observe that for each 
row of $W$, the corresponding row $(W_i \,\,|\,\,b_i) \in \mathbb{R}^{n+1}$, of the augmented matrix $(W \,\,|\,\,b) \in M_{m \times (n+1)}(\mathbb{R})$ determines a set: 

 \begin{equation} \label{eq:affsolnset}
 S_i := \{\vec{x} \in \mathbb{R}^n \,\,|\,\, (W_i\,\,|\,\,b_i) \cdot (\vec{x}\,\,|\,\,1) = 0\} \subseteq \mathbb{R}^n.
 \end{equation}  

 An {\em ordered affine solution set arrangement} in $\mathbb{R}^n$ is a finite ordered set, $\mathcal{S} = \{S_1, \ldots, S_m\}$, where each $S_i$ is the solution set to an affine-linear equation as described above in equation \eqref{eq:affsolnset}. If $W_i = {\bf 0}$, we say $S_i$ is {\em degenerate}. In this case $S_i$ is empty if $b_i \neq 0$ and $S_i$ is all of $\mathbb{R}^n$ if $b_i = 0$.  
 An  ordered affine solution set arrangement, $\mathcal{S} = \{S_1, \ldots, S_m\}$, is said to be {\em degenerate} if at least one $S_i$ is degenerate and {\em nondegenerate} otherwise. $\mathcal{S}$ is said to be {\em in general position} (aka {\em generic}) if, for all subsets $\{S_{i_1}, \ldots, S_{i_p}\} \subseteq \mathcal{S}$, it is the case that $S_{i_1} \cap \ldots \cap S_{i_p}$ is an affine-linear subspace of $\mathbb{R}^n$ of dimension $n-p$, where a negative-dimensional intersection is understood to be empty.  Note that generic implies nondegenerate.

If $S_i$ is nondegenerate, it is a hyperplane, and we will denote it by $H_i$. In this case, $\mathbb{R}^n \setminus H_i$ has two connected components,
\begin{eqnarray*} \label{eq:coorient}
	H_i^+ &:=& \{\vec{x} \in \mathbb{R}^n \,\,|\,\, (W_i \,\,|\,\,b_i) \cdot (\vec{x}\,\,|\,\,1) > 0\}\\
	H_i^- &:=& \{\vec{x} \in \mathbb{R}^n \,\,|\,\,  (W_i \,\,|\,\,b_i) \cdot (\vec{x}\,\,|\,\,1) < 0\}, \label{eq:orientations}
\end{eqnarray*}
which endows $H_i$ with a co-orientation, pointing toward $H_i^+$.  

If an ordered affine solution set $\mathcal{S}=\{S_1,\ldots,S_m\}$ arrangement is nondegenerate, we can associate to $\mathcal{S}$ an {\em ordered, co-oriented hyperplane arrangement}, $\bf{A} = \{\bf{H}_1, \ldots, \bf{H}_m\}$ in $\mathbb{R}^n$. By forgetting the ordering of the set and the co-orientations of the affine hyperplanes we obtain a classical hyperplane arrangement; that is, a finite set, $\mathcal{A} = \{H_1, \ldots, H_m\}$, of affine hyperplanes in $\mathbb{R}^n$.  If $\mathcal{S}=\{S_1,\ldots, S_m\}$ is degenerate, by first removing the degenerate solution sets from $\mathcal{S}$, we can still associate to $\mathcal{S}$ an ordered, co-oriented hyperplane arrangement $\bf{A}$ or classical hyperplane arrangement $\mathcal{A}$ -- albeit with fewer that $m$ hyperplanes.  
We shall use $\bf{A}$ (resp., ${\bf H}_i$) if the hyperplanes are equipped with a co-orientation and $\mathcal{A}$ (resp., $H_i$) if not.

The ordered affine solution set arrangement associated to a layer map, $\mathbb{R}^n \rightarrow \mathbb{R}^m$, of a neural network is the set $\mathcal{S} = \{S_1, \ldots, S_m\}$ as in equation \eqref{eq:affsolnset}. A layer map of a neural network is said to be {\em degenerate} if its associated affine solution set arrangement $\mathcal{S}$ is degenerate, and {\em nondegenerate} otherwise. A layer map of a neural network is said to be {\em generic} if the corresponding affine solution set arrangement is generic, and {\em nongeneric} otherwise.

\begin{definition} \label{defn:genericlayer} 
A neural network whose layer maps are all nondegenerate is said to be {\em nondegenerate}. A neural network whose layer maps are all generic is said to be {\em generic}.
\end{definition}

The \emph{rank} of a hyperplane arrangement $\mathcal{A}$ in $\mathbb{R}^n$ is the dimension of the space spanned by the normals to the hyperplanes in $\mathcal{A}$. 
(For more on the geometry and combinatorics of hyperplane arrangements, see \cite{Stanley}.)

The following lemmas (cf. \cite{Stanley}) follow in a straightforward way from standard facts in linear algebra.

\begin{lemma} \label{lem:GenericArrangementInvertibleMap}
Let $A:\mathbb{R}^n \to \mathbb{R}^n$ be an affine-linear map given by $A(\vec{x}) := (W|b) \vec{x}'$
and let 
$\mathcal{S} = \{S_1, \ldots, S_n\}$ be the associated affine solution set arrangement in $\mathbb{R}^n$ described in equation \eqref{eq:affsolnset}. 
Then $A$ is an invertible function if and only if $\mathcal{S}$ is generic.
\end{lemma}


\begin{lemma} \label{lem:genHAdense} 
Let $\arch$ be a network architecture. For almost every ${\bf s} \in \mathbb{R}^{D}$, the neural network $\mathcal{F}_{\bf s}$ is generic.  
\end{lemma}


\section{Polyhedral complexes} \label{sec:PolyCpx}
 We will need some basic facts about the geometry and combinatorics of convex polytopes, polyhedral sets, and polyhedral complexes. We quickly recall relevant background and terminology, referring the interested reader to \cite{Grunbaum, Grunert} for a more thorough treatment.

A {\em polyhedral set $\mathcal{P}$ in $\mathbb{R}^n$} is an intersection of finitely many  closed affine half spaces $H_1^+, \ldots,  H_m^+ \subseteq \mathbb{R}^n.$ A {\em convex polytope} in $\mathbb{R}^n$ is a bounded polyhedral set. Note that a polyhedral set is an intersection of convex sets, and hence convex.  Each region of a hyperplane arrangement is the interior of a polyhedral set.

A hyperplane $H$ in $\mathbb{R}^n$ is a {\em cutting hyperplane} of a polyhedral set $\mathcal{P}$ and is said to {\em cut} $\mathcal{P}$ if there exists $x_1, x_2 \in \mathcal{P}$ with $x_1 \in \mathcal{P} \cap H^+$ and $x_2 \in \mathcal{P} \cap H^-$.
A hyperplane $H$ in $\mathbb{R}^n$ is a {\em supporting hyperplane} of $\mathcal{P}$ and is said to {\em support} $\mathcal{P}$ if $H$ does not cut $\mathcal{P}$ and $H \cap \mathcal{P} \neq \emptyset$.

For any set $S \subset \mathbb{R}^n$, the {\em affine hull} of $S$, denoted $\mbox{aff}(S)$, is the intersection of all affine-linear subspaces of $\mathbb{R}^n$ containing $S$.  
  The \emph{dimension} of a polyhedral set is the dimension of its affine hull.

Let  $\mathcal{P}$ be a polyhedral set of dimension $n$. A subset $F \subset \mathcal{P}$ is said to be a {\em face} of $\mathcal{P}$ if either $F = \emptyset$, $F = \mathcal{P}$, or $F = H \cap \mathcal{P}$ for some supporting hyperplane of $\mathcal{P}$.  $\emptyset$ and $\mathcal{P}$ are called the {\em improper} faces of $\mathcal{P}$. All other faces are {\em proper}.
 A {\em $k$--face} of $\mathcal{P}$ is a face of $\mathcal{P}$ that has dimension $k$. 
A {\em facet} of $\mathcal{P}$ is an $(n-1)$--face of $\mathcal{P}$.
 A {\em vertex} of $\mathcal{P}$ is a $0$--face of $\mathcal{P}$. 

\begin{lemma} \cite[Sec 26]{Grunbaum} \label{lem:convpoly} Every polyhedral set $\mathcal{P}$ has an {\em irredundant} realization as an intersection $\mathcal{P} = H_1^+ \cap \ldots \cap H_m^+$ satisfying the property that $$\mathcal{P} \neq \bigcap_{j \neq i}H_j^+$$ for each $i = 1, \ldots, m$. Moreover, for an irredundant realization as above, the set of facets of $\mathcal{P}$ is precisely the set of {\em proper} faces of the form $\mathcal{P} \cap H_i$.
\end{lemma}

 A {\em polyhedral complex} $\mathcal{C}$ of dimension $d$ is a {\em finite}  set 
 of polyhedral sets of dimension $k$, for $0 \leq k \leq d$, called the \emph{cells} of $\mathcal{C}$, such that i)  If $P \in \mathcal{C}$, then every face of $P$ is in $\mathcal{C}$, and ii)  if $P, Q \in \mathcal{C}$, then $P \cap Q$ is a single mutual face of $P$ and $Q$.  The \emph{domain} or \emph{underlying set} $|\mathcal{C}|$ of a polyhedral complex $\mathcal{C}$ is the union of its cells. If $\mathcal{C}$ is a polyhedral complex embedded in $\mathbb{R}^n$ and $|\mathcal{C} | = \mathbb{R}^n$, we call $\mathcal{C}$ a {\em polyhedral decomposition of $\mathbb{R}^n$}.

A \emph{polyhedral subcomplex} of 
$\mathcal{C}$ is a subset $\mathcal{C}' \subseteq \mathcal{C}$ such that for every cell $P$ in  $\mathcal{C}'$, every face of $P$ is also in $\mathcal{C}'$.  The {\em $k$--skeleton} of $\mathcal{C}$, denoted $\mathcal{C}_k$, is the subcomplex of all polyhedral sets of $\mathcal{C}$ of dimension $i$, where $0 \leq i \leq k$. 

 Any hyperplane arrangement $\mathcal{A}$ in $\mathbb{R}^n$ induces a polyhedral decomposition,  $\CA$,  of  $\mathbb{R}^n$
as follows.  Define the $n$-dimensional cells of  $\CA$  to be the closures of the regions of $\mathcal{A}$, and for $0 < i < n$, inductively define the $i$-dimensional cells of $\CA$ to be the facets of the $i+1$ dimensional cells. 
Similarly,  any affine solution set arrangement $\mathcal{S} = \{S_1, \ldots, S_m\}$ in $\mathbb{R}^n$ induces a polyhedral decomposition, $\mathcal{C}(\mathcal{S})$, formed by first removing the degenerate affine solution sets from $\mathcal{S}$ to obtain a hyperplane arrangement, $\mathcal{A}$, and setting $\mathcal{C}(\mathcal{S}) := \CA$. 

For polyhedral complexes $M$ and $R$, a map  $f: |M| \rightarrow |R|$ is {\em cellular} if for every cell $K \in M$ there exists a cell $L \in R$ with $f(K) \subseteq L$.  For a polyhedral complex $M$ embedded in $\mathbb{R}^m$, a map $f: |M| \to \mathbb{R}^r$ is \emph{affine-linear on cells of $M$}, if for each cell $K \in M$, the restriction of $f$ to $|K|$ is  affine-linear.\footnote{Note that \cite{Grunert}, following the classical texts on PL geometry, uses ``linear on cells" rather than ``affine-linear on cells," since PL charts need not specify an origin. Since our polyhedral complexes are canonically embedded in $\mathbb{R}^n$, we use ``affine-linear" throughout.}

A polyhedral complex $M'$ is said to be a {\em subdivision} of a polyhedral complex $M$ if $|M| = |M'|$ and 
each cell of $M'$ is contained in a cell of $M$.

Of particular interest here is the \emph{level set complex}  associated to a pair of polyhedral complexes $M$ and $R$, with $R$ embedded in $\mathbb{R}^r$, and a map $f:|M| \to \mathbb{R}^r$ affine-linear on cells of $M$.  The level set complex, which is a polyhedral complex by  \cite[Lemma 2.5]{Grunert}, is the complex \[ M_{\in R} :=  \{ S  \cap f^{-1}(Y) \mid S \in M, Y \in R\}.\]

In the present work, we focus primarily on maps $\mathbb{R}^n \rightarrow \mathbb{R}$, where the single $0$--cell of the polyhedral complex $R \subseteq \mathbb{R}$ is a threshold $t \in \mathbb{R}$ and the two $1$--cells are the unbounded intervals $(-\infty, t]$ and $[t, \infty)$.

\section{Transversality} \label{sec:Trans}

 In this subsection, we recall classical transversality results and state and extend them to situations of relevance to us.

\subsection{Classical transversality results} We follow the terminology and notation of \cite{GuilleminPollack}.

Denote the tangent space of a smooth manifold $X$ at a point $x \in X$ by $T_xX$.  Recall that for a smooth map $f:X \to Y$ of manifolds with $f(x) = y$, the derivative $df_x$ is a linear map between tangent spaces, $df_x:T_xX \to T_yY$, and the image $df_x(T_xX)$ is a linear subspace of $T_yY$.  If $U$ and $V$ are two subspaces of a linear space $W$, then their sum, $U + V$, is the subspace $\{ u + v: u \in U, v \in V\}$.  

\medskip

In Definition \ref{defn:maptransverse} and Theorems \ref{thm:Transmaps} and \ref{thm:ParTrans}, assume $X$ to be a smooth manifold with or without boundary, $Y$ and $Z$ to be smooth manifolds without boundary, $Z$ a smoothly embedded submanifold of $Y$, and $f:X \to Y$ a smooth map.  

\begin{definition} \label{defn:maptransverse} We say that $f$ is {\em transverse} to $Z$ and write $f \pitchfork Z$ if 
\begin{equation} \label{eq:maptransverse}
df_p(T_pX) + T_{f(p)}Z = T_{f(p)}Y
\end{equation} 
for all $p \in f^{-1}(Z)$.
\end{definition}

 Note that Definition \ref{defn:maptransverse} allows for the possibility that $X$ is a manifold of dimension $0$, i.e. consists of--without loss of generality--a single point $p$.   In this case $T_p\{p\} = \{0\}$ and so $df_p(T_p\{p\}) = \{0\}$, so condition \eqref{eq:maptransverse} reduces to the condition that if $f(p) \in Z$, then $Z$ and $Y$ must agree in a neighborhood of $f(p)$. Note also
that if the image $f(X)$ does not intersect $Z$, then condition \eqref{eq:maptransverse} is vacuously true, hence $f \pitchfork Z$.

\begin{theorem}[Map Transversality Theorem]  \cite[p.~28]{GuilleminPollack} \label{thm:Transmaps} If $f$ is transverse to $Z$, then $f^{-1}(Z)$ is an embedded submanifold of X. Furthermore, the codimension of $f^{-1}(Z)$ in $X$ equals the codimension of $Z$ in $Y$.  
\end{theorem}

The Map Transversality Theorem uses the standard convention that the empty set can assume any dimension.  If $f^{-1}(Z) = \emptyset$, one considers the codimension of $f^{-1}(Z)$ in $X$ to be the codimension of $Z$ in $Y$.

\begin{theorem}[Parametric Transversality Theorem] \cite[p.~68]{GuilleminPollack} \label{thm:ParTrans} Let $S$ be a smooth manifold and let $F: X\times S\rightarrow Y$ be a smooth map.  If  $F$ is  transverse to $Z$, then for (Lebesgue) almost every $s\in S$ the restriction map $F_s: X \to Y$ given by $F_s(x) = F(x,s)$ is transverse to $Z$.
\end{theorem}

We wish to apply Theorem \ref{thm:ParTrans} to the parametrized family of neural networks of a fixed architecture (Definition \ref{ParameterizedNN}), but this family does not satisfy the smoothness requirements, so we develop the necessary non-smooth analogues in \S \ref{ss:nonsmoothtransversality}.

\subsection{Extensions of the classical transversality results to maps on polyhedral complexes that are smooth on cells} \label{ss:nonsmoothtransversality} \

We introduce a polyhedral analogue of Definition \ref{defn:maptransverse}:

\begin{definition} \label{defn:polytransverse} Let $X$ be a polyhedral complex of dimension $d$ in $\mathbb{R}^n$, let $f: |X| \rightarrow \mathbb{R}^r$ be a map  which is smooth  on all cells of $X$
 and let $Z$ be a  smoothly embedded submanifold (without boundary) of $\mathbb{R}^r$. We say that $f$ is  \emph{transverse on cells} to $Z$ and write $f \pitchfork_c Z$ if: 
\begin{enumerate}
	\item the restriction of $f$ to the {\em interior}, $\mbox{int}(C)$, of every $k$--cell $C$ of $X$ is transverse to $Z$ (in the sense of Definition \ref{defn:maptransverse}) when $1 \leq k \leq d$, and 
	\item \label{item:0celltransversereq} the restriction of $f$ to every $0$--cell of $X$ is transverse to $Z$. 
\end{enumerate}
\end{definition}

Note that a function defined on a $0$-cell is considered to be smooth. We will be particularly interested in the case in which $r=1$ and $Z =\{t\}$ is a threshold in $\mathbb{R}$.  

\begin{corollary} \label{cor:codimtrans} Let $X$ be a polyhedral complex of dimension $d$ in $\mathbb{R}^n$.  Let $f : |X| \rightarrow \mathbb{R}^r$ be a map which is smooth on cells of $X$ and let $Z$ be a smoothly  embedded submanifold of $\mathbb{R}^r$ for which $f \pitchfork_c Z$. Then we have:
\begin{itemize}
	\item For every $k$--cell $C\in X$, where $1 \leq k \leq d$, $f^{-1}(Z) \cap \mbox{int}(C)$ is a (possibly empty) smoothly embedded submanifold of $\textup{int}(C)$.  Furthermore, the codimension of $f^{-1}(Z) \cap \mbox{int}(C)$ in $\mbox{int}(C)$ equals the codimension of $Z$ in $\mathbb{R}^r$. 
	\item If $\textrm{dim}(Z) < r$, then for every $0$--cell $C \in X$ (vertex), $f(C) \not \in Z$. 
\end{itemize}
\end{corollary}

\begin{proof}This follows immediately from Theorem~\ref{thm:Transmaps}, since the interior of any polyhedral set of dimension $k \in \mathbb{N}$ is a nonempty smooth manifold. Note also that condition \eqref{item:0celltransversereq} of Definition \ref{defn:polytransverse} implies that if $f$ is transverse on cells to $Z$ and there exists a vertex $v$ of $X$ such that $f(v) \in Z$, then $Z$ must have the full dimension $r$.  Thus, if $\textup{dim}(Z) < r$, $f$ being transverse on cells to $Z$ implies no vertex of $X$ is sent by $f$ to $Z$. 
\end{proof}

We will need the following version of Theorem \ref{thm:ParTrans} for families of maps that are smooth on cells of a polyhedral complex.

\begin{proposition} \label{prop:ParTransPolyCpx} Let $X$ be a polyhedral complex in $\mathbb{R}^n$, 
 $S$ a smooth manifold without boundary, and $Z \subseteq \mathbb{R}^r$ a smoothly embedded submanifold without boundary. 
Let $F: |X| \times S \to \mathbb{R}^n$ be a map such that for each cell $C \in X$, the restricted map $F\vert_{C \times S} : C \times S \to \mathbb{R}^r$ is smooth and the further restricted map  $F\vert_{C' \times S}: C' \times S \rightarrow \mathbb{R}^r,$
where 
\[C' =  \begin{cases} 
\textrm{int}(C) & \textrm{ if } $C$ \textrm{ is of dimension } \geq 1,\\
C  & \textrm{ if } $C$ \textrm{ is of dimension } 0, \\
 \end{cases} \]
is  transverse to $Z$.
Then for (Lebesgue) almost every $s \in S$, the map \[f_s: |X| \rightarrow \mathbb{R}^r\] given by $f_s(x) = F(x,s)$ is transverse on cells to $Z$.
\end{proposition}

\begin{proof} For each cell $C \in X$, the Parametric Transversality Theorem implies that there exists a null set $S_C \subset S$ such that $f_s\vert_{C'}$ is transverse to $Z$ for every $s \in S \setminus S_C$.  Let $S_X = \bigcup_{C \in X} S_C$; as a finite union of null sets, $S_X$ is a null set.  
Then for every $C \in X$ and $s \in S \setminus S_X$, we have that  $f_s\vert_{C'}:C' \rightarrow \mathbb{R}^r$ 
is transverse to $Z$. Hence $f_s$ is transverse on cells to $Z$ for all $s \in S \setminus S_X$. 
\end{proof}

\section{Maps on polyhedral complexes and transversal thresholds} \label{sec:PolyMapsTransThresh}
 We now turn to applying the transversality statements developed in the previous section to ReLU neural network maps.  

\begin{definition} \label{def:transversethreshold}  Let $M$ be a polyhedral complex embedded in $\mathbb{R}^{n_0}$, $n_0 \in \mathbb{N}$, and let $F: |M| \rightarrow \mathbb{R}$ be a map that is smooth on cells. 
A threshold $t \in \mathbb{R}$ is said to be {\em transversal} 
for  $F$ and $M$ 
if $F$ is  transverse on cells (Definition  \ref{defn:polytransverse}) to the submanifold $\{t\} \subseteq \mathbb{R}$.  
In this case, we write $F \pitchfork_c  \{t\}$. 
\end{definition}

Although Section \ref{ss:nonsmoothtransversality} and Definition \ref{def:transversethreshold} require only that the map $F$ be smooth on cells, from this point onwards we restrict to the case that $F$ is affine-linear on cells, since this is the setting relevant for understanding ReLU neural network maps.  For the remainder of this section, let $M$ be a polyhedral complex embedded in $\mathbb{R}^{n_0}$, $n_0 \in \mathbb{N}$, and let $F: |M| \rightarrow \mathbb{R}$ be a map that is affine-linear on cells.

\begin{definition} \label{defn:nonconstantstar} 
A point $x \in M$ is said to have a {\em $F$-nonconstant  cellular neighborhood in $M$}  if $F$ is nonconstant on each cell of $M$ containing $x$.
\end{definition}

 Note that each vertex of $M$ is itself a cell on which $F$ is necessarily constant; hence, no vertex of $M$ can be said to have a $F$-nonconstant cellular neighborhood.

\begin{lemma} \label{lem:nonconscellneighborhood}
A threshold $t \in \mathbb{R}$ is transversal for $F$ and $M$ if and only if each point $p \in F^{-1}(\{t\})$ has a  $F$-nonconstant cellular neighborhood in $M$.
\end{lemma}

\begin{proof} 
The threshold $t$ is transversal for $F$ and $M$ if and only if for any $k$-cell $C \in X$ with $k \geq 1$ the restriction of $f$ to $\textup{int}(C)$ is transverse to $\{t\}$, and the restriction of $f$ to any $0$-cell $X$ is transverse to $\{t\}$. This is equivalent to the statement that if $p \in F^{-1}(\{t\})$ and $p$ is in a cell $C \in M$, then 
\[df_p(T_pC) + T_{f(p)}\{t\} = T_{f(p)}\mathbb{R}.\]
  Since $ T_{f(p)}\{t\} = \{0\}$, this equality holds if and only if $df_p(T_pC) = T_{f(p)} \mathbb{R} \cong \mathbb{R}$ for every cell $C$ containing $p$. By Corollary \ref{cor:codimtrans}, this is equivalent to $p$ having a nonconstant cellular neighborhood, as desired. 
\end{proof}

\begin{lemma} \label{lem:genthreshconvpoly} 
Let $t \in \mathbb{R}$ be a transversal threshold  for $F$ and $M$. Then for every cell $C \in M$, $F^{-1}(\{t\}) 
\cap C$ is either empty or
$\textrm{aff}(F^{-1}(\{t\}) \cap C)$ is a hyperplane in $\textrm{aff}(C)$.
 Moreover, whenever $F^{-1}(\{t\}) \cap C$ is nonempty, the hyperplane $\mbox{aff}(F^{-1}(\{t\}) \cap C)$ 
 cuts $C$. 
\end{lemma}

\begin{proof}
The statement that $F^{-1}(\{t\}) \cap C$ is a submanifold of codimension $1$ in $C$ is from the Map Transversality Theorem (Theorem \ref{thm:Transmaps}); its affine hull is a hyperplane because $F$ is affine-linear.  Let $H = \textup{aff}(F^{-1}(\{t\}) \cap C) \neq \emptyset$.    If $H$ were a supporting hyperplane of $C$, then $H \cap C$ would be a non-empty lower-dimensional face of $C$, all of whose points map to $t$. Applying Lemma \ref{lem:nonconscellneighborhood}, this would contradict the assumption that $t$ is a transversal threshold. Hence, $H$ cuts $C$ whenever $F^{-1}(\{t\}) \cap C \neq \emptyset$.
\end{proof}

\begin{lemma} \label{lem:transversalthreshold}  All but finitely many thresholds $t \in \mathbb{R}$ are transversal for $F$ and $M$. 
\end{lemma}

\begin{proof} The polyhedral complex $M$ is, by definition, finite.  Hence there are only finitely many cells on which $F$ is constant. But Lemma \ref{lem:nonconscellneighborhood} tells us that the images of the constant cells are the only nontransversal thresholds for $F$ and $M$. 
\end{proof}

\section{Bent hyperplane arrangements and canonical polyhedral complexes} \label{sec:BHA}

The following notion was introduced in \cite{HaninRolnick}.

\begin{definition}\cite[Eqn. (2), Lem. 4]{HaninRolnick} \label{defn:benthyperarr} Let 
  \[\mathbb{R}^{n_0}  \xrightarrow{F_1 = \sigma \circ A_1} \mathbb{R}^{n_1}  \xrightarrow{F_2 = \sigma \circ A_2}
   \ldots \xrightarrow{F_m = \sigma \circ A_{m}} \mathbb{R}^{n_m}  \xrightarrow{G = A_{m+1}} \mathbb{R}^1\]
\noindent be a ReLU neural network and let $\mathcal{A}^{(k)} = \left\{H_1^{(k)}, \ldots, H_{n_{i_k}}^{(k)}\right\}$ denote the hyperplane arrangement in $\mathbb{R}^{n_{k-1}}$ associated to the layer map $F_k$.  A {\em bent hyperplane associated to the $k^{\textrm{th}}$ layer of $F$}, for $k \in  \{2,\ldots,m\}$,  is the preimage in $\mathbb{R}^{n_0}$ 
of any hyperplane $H_i^{(k)} \subseteq \mathbb{R}^{n_{k-1}}$ associated to the $k^{\textrm{th}}$ layer map: \[(F_{k-1} \circ \ldots \circ F_1)^{-1}\left(H_i^{(k)}\right).\]  Although they do not bend, for the sake of consistency we will refer to the hyperplanes $H_i^{(1)}$  in $\mathcal{A}^{(1)} \subseteq \mathbb{R}^{n_0}$ as the {\em bent hyperplanes associated to the $1^{\textrm{st}}$ layer of $F$}. 
\end{definition}

We will denote by $\mathcal{B}_F^{(k)}$ the union of the bent hyperplanes associated to the $k^{\textrm{th}}$ layer of $F$ and refer to this union as the {\em bent hyperplane arrangement associated to the $k^{\textrm{th}}$ layer of $F$}.

We will denote by $\mathcal{B}_F := \bigcup_{k=1}^{m} \mathcal{B}_F^{(k)}$ the union of all bent hyperplanes from all $m$ layers of $F$ and refer to this union as the {\em bent hyperplane arrangement} associated to $F$.

\begin{remark} \label{rem:noGinBHA}
Since there is no activation function on the final layer map to induce any additional loci of non-differentiability, we do not include its associated bent hyperplane in $\mathcal{B}_F$.
\end{remark}

It is immediate that $F$ is smooth on the complement of $\mathcal{B}_F$.

\begin{definition}  \label{defn:activationregion} \cite[Def 1, Lem 2]{HaninRolnick} Let $F: \mathbb{R}^{n_0} \rightarrow \mathbb{R}$  be a 
ReLU neural network map. An {\em activation region} of $F$ is a connected component of the complement of the bent hyperplane arrangement associated to $F$, i.e. a connected component of $\mathbb{R}^{n_0} \setminus \mathcal{B}_F$.
\end{definition}

\begin{remark}
Note that for $k \geq 2$ it is possible for $\mathcal{B}_F^{(k)}$ to have codimension $0$, not $1$, in $\mathbb{R}^{n_0}$.  As a simple example of this phenomenon, consider a two-layer ReLU neural network 
\[F: \xymatrix@1{\mathbb{R}^2 \ar[r]^{F_1} & \mathbb{R}^2 \ar[r]^{F_2} & \mathbb{R} \ar[r]^G & \mathbb{R}},\] where $\mathcal{A}^{(1)}$ is the standard coordinate hyperplane arrangement, and $\mathcal{A}^{(2)} = \left\{H^{(2)}\right\},$ where $H^{(2)}$ is any line through the origin with negative slope. Then \[\mathcal{B}_F^{(1)} = \{H_1^{st}, H_2^{st}\},\] the standard co-oriented coordinate axes, and \[\mathcal{B}_F^{(2)} = \{(x,y) \in \mathbb{R}^2 \,\,|\,\,x,y \leq 0,\}\]  the closed non-positive orthant. In particular, the bent hyperplane arrangement is codimension $0$, not $1$, and hence the closure of the activation regions (Definition \ref{defn:activationregion}) is a proper subset of $\mathbb{R}^n$. 

This phenomenon arises when a map fails to be transversal to a threshold, an observation that motivates Definition \ref{defn:GenericBHA} and Theorem \ref{thm:nminus1skel}. Note that it is also a measure zero phenomenon. See Theorem  \ref{thm:transversalgeneric}.
\end{remark}

 We now define a canonical polyhedral decomposition of the domain of a ReLU neural network.
   In the transversal case, we explicitly relate this decomposition to the bent hyperplane arrangements and activation regions in Theorems \ref{thm:BHAispolycpx} and \ref{thm:nminus1skel}.

\begin{definition} \label{def:canonicalpolyedraldecomp}
Let \[F: \xymatrix@1{\mathbb{R}^{n_0} \ar[r]^{F_1} & \mathbb{R}^{n_1} \ar[r]^{F_2} & \ldots \ar[r]^{F_m} & \mathbb{R}^{n_m} \ar[r]^G & \mathbb{R}}\] 
be a ReLU neural network. 
 For $i \in \{1,\ldots,m\}$, denote by $R^{(i)}$ the polyhedral complex in $\mathbb{R}^{n_{i-1}}$ induced  by the hyperplane arrangement associated to the $i^{\textrm{th}}$ layer map $F_i$.  Inductively define polyhedral complexes, $\mathcal{C}(F_i \circ \ldots  \circ F_1)$, in $\mathbb{R}^{n_0}$ as follows: 
Set 
\begin{itemize}
	\item $\mathcal{C}(F_1):=R^{(1)},$ and 
	\item $ \Ci:= \Cii_{\in R^{(i)}}$ for $i = 2, \ldots, m$.
\end{itemize}
  The \emph{canonical polyhedral complex} associated to $F$ is  $\mathcal{C}_F:= \mathcal{C}(F_m \circ \ldots \circ F_1).$
\end{definition}

\begin{theoremBHAispolycpx}
Let \[F: \xymatrix@1{\mathbb{R}^{n_0} \ar[r]^{F_1} & \mathbb{R}^{n_1} \ar[r]^{F_2} & \ldots \ar[r]^{F_m} & \mathbb{R}^{n_m} \ar[r]^G & \mathbb{R}}\] be a ReLU neural network. 
For each $i$, $\Ci$ is a polyhedral decomposition of $\mathbb{R}^{n_0}$ satisfying
\begin{enumerate}
	\item $F_i \circ \ldots \circ F_1$ is affine-linear on the cells of $\Ci$,
	\item  $\bigcup_{k=1}^{i} \mathcal{B}_F^{(k)}$  is the domain of  a polyhedral subcomplex of $\Ci$. 
\end{enumerate}
\end{theoremBHAispolycpx}

\begin{proof}
For each $i = 1,\ldots,m$, denote by $\mathcal{A}^{(i)}$ the hyperplane arrangement 
associated to the layer map $F_i: \mathbb{R}^{n_{i-1}} \rightarrow \mathbb{R}^{n_i}$ 
 and denote by $R^{(i)}$ the induced polyhedral decomposition of $\mathbb{R}^{n_{i-1}}$.  We proceed by induction on $i$.  
 
For $i=1$, it is immediate that $\mathcal{B}_F^{(1)} = {\mathcal{A}}^{(1)}$ forms the $(n_0-1)$--skeleton of $R^{(1)}$ and $F_1$ is affine-linear on cells of $\mathcal{C}(F_1) = R^{(1)}$.

Now consider $i>1$ and assume the statement holds for $i-1$.   Since level set complexes are polyhedral complexes, condition (i) of the inductive hypothesis implies $\Ci$ is a polyhedral complex.

By condition (i) of the inductive hypothesis, each cell in $\Ci$ 
is the intersection of a cell in $\Cii$ with the preimage of a cell in $R^{(i)}$.  The map $F_{i-1} \circ \ldots \circ F_1$ is affine-linear on each such intersection by assumption. The layer map $F_i:\mathbb{R}^{n_{i-1}} \to \mathbb{R}^{n_i}$ is affine-linear on cells of $R^{(i)}$.  Condition (i) follows.  

By condition (ii) of the inductive hypothesis, $\bigcup_{k=1}^{i-1} \mathcal{B}_F^{(k)}$ is the domain of a polyhedral subcomplex of $\Cii$. By definition, $\Ci$ is a subdivision of $\Cii$, so $\bigcup_{k=1}^{i-1} \mathcal{B}_F^{(k)}$ is the domain of a polyhedral subcomplex of $\Ci$.  Let $R^{(i)}_{(n_{i-1}-1)}$ denote the $(n_{i-1}-1)$--skeleton of $R^{(i)}$. Noting that the domain of $R^{(i)}_{(n_{i-1}-1)}$  is the union of the hyperplanes in $\mathcal{A}^{(i)}$, we have
 \[\left|\Cii\right|_{\in R^{(i)}_{(n_{i-1} - 1)}} = \mathcal{B}_F^{(i)}.\] 
  Since the union of two subcomplexes of a polyhedral complex is a subcomplex, $\bigcup_{k=1}^{i} \mathcal{B}_F^{(k)}$ is a polyhedral subcomplex of $\Ci$, implying condition (ii).

\end{proof}

The following definition is important for the constructions in Section \ref{sec:TransNN}:

\begin{definition} \label{def:NeuralNetworkTransversalThreshold}
A threshold $t \in \mathbb{R}$ is a \emph{transversal threshold} for a neural network $F: \mathbb{R}^{n_0} \rightarrow \mathbb{R}$ 
if $t$ is a transversal threshold for $F$ and its canonical polyhedral complex $\mathcal{C}(F)$.
\end{definition}

\section{Piecewise smoothness of the parametrized family of neural networks} \label{sec:PWSmooth}

Throughout this section, consider any fixed architecture $\arch$ 
and let  $\mathcal{F}: \mathbb{R}^{n_0} \times \mathbb{R}^D \to \mathbb{R}$ be the parametrized family of neural networks of this architecture (Definition \ref{ParameterizedNN}).  

\begin{lemma}  \label{l:piecewisesmoothfuncts}
There exists a finite set $E$ of polynomials in the variables $x_1,\ldots,x_{n_0}, s_1, \ldots s_D$ such that the following hold:
\begin{enumerate}
\item If $T$ is a term of a polynomial in $E$, then $T$ has the form 
$x s_1^{\tau_1} \ldots s_D^{\tau_D}$
for some $x \in \{x_1,\ldots,x_{n_0},1\}$ and $(\tau_1, \ldots, \tau_{D}) \in \{0,1\}^{D}$.
\item  $\mathcal{F}$ is smooth on the complement of the set $Z$ defined by
\[ Z := \{(x,s) \in \mathbb{R}^{n_0} \times \mathbb{R}^{D}  : f_i \left((x,s)\right)=0 \textrm{ for some } f_i \in E\}. \]
\end{enumerate}
\end{lemma}

\begin{proof}
The idea is to let $E$ be the set of all possible ``inputs'' of any ReLU in the expression defining $\mathcal{F}$.  Rather than presenting a formal proof, we give an illustrative example that demonstrates all the key ideas.

We consider the network architecture $(1,2,1;1)$.  The associated parametrized family is the map $\mathcal{F}: \mathbb{R} \times \mathbb{R}^9 \to \mathbb{R}$ given by 
  \[ (x,(a,b,c,d,e,f,g, h, i)) \mapsto  h  \cdot \textup{ReLU}( e \cdot \textup{ReLU}(ax+b) + f \cdot \textup{ReLU}(cx+d) + g) + i \]
 Each of the three ReLU's in this expression acts as either the identity or $0$, depending on the sign of its argument.  
 Let $E$ be the set of all possible expressions that are inputs of a ReLU in the expression for $\mathcal{F}$, allowing for the possibility that each nested ReLU could be either $0$ or the identity.  That is, 
 \[E = \{ax+b, cx+d, e(ax+b) + f (cx+d) + g, e(ax+b) + g, f (cx+d) + g, g   \}. \]
   Let $Z \subset  \mathbb{R} \times \mathbb{R}^9$ be the set of points $(x,s)$ where at least one function in $E$ evaluates to $0$.

For any fixed input $(x,s)$, $\mathcal{F}((x,s))$ is given by one of $2^3$ possible (not necessarily distinct) formulas (which correspond to each of the $3$ ReLU's being in one of two possible ``states'').
 Let $H$ be the set of $2^3$ (not necessarily distinct) functions formed by replacing each ReLU with either $0$ or the identity.  
Which of these expressions represents $\mathcal{F}$ locally only changes at points where the argument of a ReLU in the expression for $\mathcal{F}$ -- i.e. a polynomial in $E$ -- changes sign. \noindent  Now, since all the functions in $E$ are continuous and $Z$ is their set of zeros, for any point $(x,s) \in (\mathbb{R} \times \mathbb{R}^9) \setminus Z$  there exists a neighborhood $U$  of $(x,s)$ on which the sign of each function in $E$ is constant.  Consequently, there is a fixed function $f \in H$ such that $\mathcal{F}$ agrees with $f$ on $U$.  Since all functions in $H$ are smooth, it follows that the restriction of  $\mathcal{F}$  to $(\mathbb{R} \times \mathbb{R}^9) \setminus Z$ is smooth.

  \end{proof}

\begin{theoremParametrizedPiecewiseSmooth}
There exists an algebraic set $Z \subset \mathbb{R}^{n_0} \times \mathbb{R}^{D}$ such that 
\begin{enumerate}
\item $\mathcal{F}$  \label{i:1}
  is smooth on the complement of $Z$,
\item \label{i:2} $Z$ is the vanishing set of a polynomial, and hence is a closed, nowhere dense subset with Lebesgue measure $0$, and 
\item \label{i:3} the complement of $Z$ consists of finitely many connected components.  
\end{enumerate}
\end{theoremParametrizedPiecewiseSmooth}

\begin{proof}
Let $E$ be the set of polynomials constructed in Lemma  \ref{l:piecewisesmoothfuncts}, define the polynomial $F := \prod_{f_i \in E} f_i$, and observe that the set $Z \subset \mathbb{R}^{n_0 + D}$ from Lemma \ref{l:piecewisesmoothfuncts} is the vanishing set for the polynomial $F$.  Items \eqref{i:1} and \eqref{i:2} follow immediately.  
To see (iii),  consider the polynomial $G: \mathbb{R}^{D+n_0} \times \mathbb{R} \to \mathbb{R}$  defined by $G(x,z) = z F(x) - 1$, and let $Z_G \subset \mathbb{R}^{D+n_0 }\times \mathbb{R}^1$ be the vanishing set for $G$. 
As the vanishing set of a real polynomial, $Z_G$ is a real algebraic variety; hence $Z_G$ has finitely many connected components (by e.g.
 \cite[Theorem 3]{whitney1992elementary} or \cite[Lemma 2.5]{Warren}).  
Note that $G(x,z) = 0$ if and only if
$F(x) \neq 0$, $z = 1/F(x)$.  Hence $Z_G$ is homeomorphic (via $(x,z) \mapsto x$) to 
 $\mathbb{R}^{D+n_0} \setminus Z$; consequently  $\mathbb{R}^{D+n_0} \setminus Z$ also has finitely many connected components.  

\end{proof}

\begin{proposition} \label{prop:benthyperplanesINZF}
 Let $Z$ be the algebraic set constructed in the proof of Lemma~\ref{l:piecewisesmoothfuncts}.
 For each $s \in  \mathbb{R}^{D}$, denote by $\mathcal{F}_s$ the neural network map defined by $\mathcal{F}_s(x) = \mathcal{F}(x,s)$.  
 Then $\bigcup_{(x,s) \in Z} \left(\mathcal{B}_{\mathcal{F}_s} \times \{s\} \right) \subseteq Z.$
\end{proposition}

\begin{proof} 
Fix a parameter $s \in  \mathbb{R}^{D}$.  The bent hyperplane arrangement $\mathcal{B}_{\mathcal{F}_s}$ is the union of the preimages in $\mathbb{R}^{n_0}$ of the hyperplanes in the hyperplane arrangements associate to the layer maps of $\mathcal{F}_s$.  Each such hyperplane is the vanishing set of the argument of a ReLU in the expression for $\mathcal{F}_s$.  Thus the polynomial 
 that defines such a hyperplane is obtained by substituting the values for $s$ into one of the polynomials in the set $E$ constructed in the proof of Lemma~\ref{l:piecewisesmoothfuncts}, implying $F(x,s) = 0$, where $F$ is the product of the polynomials in $E$.  Hence, any point $x$ in the preimage in $\mathbb{R}^{n_0}$ of such a hyperplane satisfies $(x,s) \in Z$.

\end{proof}

\section{Transversal neural networks} \label{sec:TransNN}

In what follows, let $\pi_j: \mathbb{R}^m \rightarrow \mathbb{R}$ denote the projection onto the $j^{\textrm{th}}$ coordinate.
For any neural network map $$\mathbb{R}^{n_0}  \xrightarrow{F_1 = \sigma \circ A_1} \mathbb{R}^{n_1}  \xrightarrow{F_2 = \sigma \circ A_2} 
\ldots \xrightarrow{F_m = \sigma \circ A_{m}} \mathbb{R}^{n_m}  \xrightarrow{G = A_{m+1}} \mathbb{R}^1,$$ $i \in \{1, \ldots, m\}$ and $j \in \{1, \ldots, n_{i}\}$,
the {\em node map}, $F_{i,j}$, is the map  \[\pi_j \circ A_i \circ F_{i-1} \circ \ldots \circ F_1 : \mathbb{R}^{n_0} \rightarrow \mathbb{R}.\]


\begin{definition} \label{defn:GenericBHA} A ReLU neural network \[F: \xymatrix@1{\mathbb{R}^{n_0} \ar[r]^{F_1} & \mathbb{R}^{n_1} \ar[r]^{F_2} & \ldots \ar[r]^{F_m} & \mathbb{R}^{n_m} \ar[r]^G & \mathbb{R}}\] is said to be {\em transversal} if,  
for each $i \in \{1, \ldots, m\}$ and each $j \in \{1, \ldots, n_i\}$, $t=0$ is a transversal threshold (Definition  \ref{def:NeuralNetworkTransversalThreshold}) for the node map \[F_{i,j}: \mathbb{R}^{n_0} \rightarrow \mathbb{R}.\]

\end{definition}

\begin{remark} \label{rmk:genvstrans}
The descriptors {\em generic} and {\em transversal}, when applied to ReLU neural networks, are similar but complementary concepts.

A ReLU neural network is {\em generic} if each solution set arrangement for each layer map is generic. It is not immediate, yet it is true, that if a solution set arrangement is generic then each solution set in the arrangement intersects each {\em intersection} of solution sets in that layer transversely. 

In contrast, if a ReLU neural network is {\em transversal} then it follows from the definitions that each bent hyperplane intersects the bent hyperplanes from all previous layers transversely.

Put simply, when applied to ReLU neural networks, the term {\em generic} describes intersections of cells associated to a single layer map, and the term {\em transversal} describes intersections of cells associated to different layer maps.
\end{remark}

\begin{theoremnminus1skel}
 If a ReLU neural network  $F: \mathbb{R}^{n_0} \rightarrow \mathbb{R}$ 
 is transversal,  then the bent hyperplane arrangement $\mathcal{B}_F$ is the domain of the $(n_0-1)$--skeleton of the canonical polyhedral complex $\mathcal{C}(F)$, and the closures of the activation regions of $F$ are the $n_0$--cells of $\mathcal{C}(F)$. 
\end{theoremnminus1skel}

\begin{proof} We proceed by induction on $m$, the number of hidden layers of $F$. The base case $m=1$ is immediate. 
Now consider any fixed value of $m > 1$ and assume the result holds for all smaller values of $m$.  In particular, for each node map $F_{m,j}$ 
the bent hyperplane arrangement 
$\mathcal{B}_{F_{m,j}}  = \bigcup_{i=1}^{m-1} \mathcal{B}_{F_{m,j}}^{(i)}$ is the domain of the $(n_0-1)$--skeleton of $\mathcal{C}(F_{m-1} \circ \ldots F_1)$. 
 But for each $i \in \{1,\ldots,m-1\}$ and $j \in \{1,\ldots,n_m\}$, we have 
$\mathcal{B}_{F_{m,j}}^{(i)} = \mathcal{B}_{F}^{(i)},$
so 
the bent hyperplane arrangement,
\[\mathcal{B}_F' := \bigcup_{i=1}^{m-1} \mathcal{B}_F^{(i)},\] for the first $m-1$ layers is the $(n_0-1)$--skeleton of the polyhedral complex $\Cmi$.  

To see that $\mathcal{B}_F$ is contained in the $(n_0-1)$--skeleton of $\Cm$, we begin by noting that $\mathcal{B}_F'$ is contained in the $(n_0-1)$--skeleton of $\Cm$ since $\Cm$ is a subdivision of $\Cmi$. 
 Moreover, by definition \[\mathcal{B}_F = \mathcal{B}_F'  \cup \bigcup_{j=1}^{n_m} F_{m,j}^{-1}(\{0\}).\]
  It therefore suffices to show that  $\bigcup_{j=1}^{n_m} F_{m,j}^{-1}(\{0\})$ is contained in the $(n_0-1)$--skeleton of $\Cm$.

But since $0$ is a transversal threshold for each $F_{m,j}:\mathbb{R}^{n_0} \to \mathbb{R}$, this follows from Corollary \ref{cor:codimtrans}. Explicitly, for every cell $C \in \Cmi$, Theorem \ref{thm:Transmaps} tells us that 
$C \cap F_{m,j}^{-1}(\{0\})$ is codimension $1$ in $C$.  Since $\Cmi$ has dimension $n_0$, it follows that any new cell in $\mathcal{B}_F \setminus \mathcal{B}_F'$ has dimension $\leq (n_0-1)$, as desired.

 To see that $\Cm$ is contained in $\mathcal{B}_F$ we will show that any $k$--cell $C$ in $\Cm_{n_0-1}$ is also in $\mathcal{B}_F$. 
Since $\Cm$ is, by definition, a subdivision of $\Cmi$, the cell $C$ is contained in a cell $C'$ of $\Cmi$. Assume WLOG that $C'$ has minimal dimension among all cells in $\Cmi$ containing $C$, and let $k'$ be the dimension of $C'$. If $k' \leq n_0-1$, then $C \subseteq C' \subseteq \Cmi_{n_0-1}$, and the inductive hypothesis tells us $C \subseteq \mathcal{B}_F'$, as desired.

So we may assume that $C'$ has dimension $n_0$. Therefore the construction described in the proof of Theorem \ref{thm:BHAispolycpx} and the fact that $C$ has dimension $\leq n_0-1$ (see also Lemma \ref{lem:genthreshconvpoly}) tells us that $C$ is equal to the intersection of $C'$ with $F_{m,j}^{-1}(\{0\})$ for some node $j$ in the $m$th layer map. It follows that $C \subseteq \mathcal{B}_F$. We conclude that $\mathcal{B}_F = \Cm_{n-1}$ as desired.
\end{proof}

\begin{theoremtransversalgeneric} For any given architecture $\arch$ 
of feedforward ReLU neural network, almost every (with respect to Lebesgue measure on $\mathbb{R}^{D}$) choice of parameters yields a generic, transversal ReLU neural network.
\end{theoremtransversalgeneric}

\begin{proof} 
Lemma \ref{lem:genHAdense} tells us that almost every choice of parameters yields a generic ReLU neural network. It therefore suffices to prove that almost every choice of parameters yields a transversal ReLU neural network.
We proceed by induction on $m$, the number of hidden 
layers. 
In the base case $m = 0$, for any network map $F$ of this architecture, the unique node map $F_{1,1}$ is the affine-linear map $G:\mathbb{R}^{n_0} \to \mathbb{R}$, and every such map is  transversal.


Now consider a fixed value of $m \geq 1$, 
and assume the result holds for smaller values of $m$.   
For any neural network $F$ of architecture $\arch$ and any $j \in \{1,\ldots,n_m\}$, the node map $F_{m,j}$ is a neural network of architecture $(n_0,\ldots,n_{m-1};1)$.  Note that such a neural network $F$ is transversal if and only if the node map $F_{m,1}$ 
is a transversal neural network and for every $j \in \{1,\ldots,n_m\}$ the node map $F_{m,j}$ has $t=0$ as a transversal threshold.  

For any $n \in \mathbb{N}$, let $\lambda_n$ denote Lebesgue measure on $\mathbb{R}^n$.  
By the inductive hypothesis, there exists a $\lambda_{D(n_0,\ldots,n_{m-1};1)}$-null set $N_1 \subset \mathbb{R}^{D(n_0,\ldots,n_{m-1};1)}$ such that for every parameter in $\mathbb{R}^{D(n_0,\ldots,n_{m-1};1)} \setminus N_1$, the associated neural network map (which we will use as the node map $F_{m,1}$) is transversal.  
 Note, furthermore, that transversality does not depend on the final affine-linear map to $\mathbb{R}^1$, so the set $N_1$ has a product structure in which each coordinate corresponding to this final affine map corresponds to a factor of $\mathbb{R}$.  For any parameter $s \in \mathbb{R}^{D(n_0,\ldots,n_{m-1};1)}$, denote by $\tilde{s}$ the projection of $s$ onto the coordinates of $s$ that determine the non-affine-linear layer of the associated network, i.e. $\tilde{s}$ is  the first $k:=
 D(n_0,\ldots,n_{m-1};1) - (n_{m-1}+1)$ coordinates of $s$.
 

 Set $\delta := D(n_0,\ldots,n_{m};1) - k$.  
We will show that for each parameter $s \in \mathbb{R}^{D(n_0,\ldots,n_{m-1};1)} \setminus N_1$, there exists a set $Y_s \subset \mathbb{R}^\delta$ such that $\lambda_{\delta}(Y_s) = 0$ and for every $w \in \mathbb{R}^\delta \setminus Y_s$, the neural network $F$ associated to the parameter $(\tilde{s},w) \in \mathbb{R}^{D\arch}$ is such  that for every $j \in \{1,\ldots,n_m\}$, $0$ is a transversal threshold of the node map $F_{m,j}$.  Assuming such sets $Y_s$ are defined, set 
$$N_2:= \left \{(\tilde{s},x) : s \in \mathbb{R}^{D(n_0,\ldots,n_{m-1};1)} \setminus N_1, x \in Y_s \right \}.$$
Tonelli's Theorem will then imply that $\lambda_{D(n_0,\ldots,n_{m};1)}(N_2) = 0$.  Then, for every parameter in $\mathbb{R}^{D\arch} \setminus N_2$, the associated neural network is transversal.

So fix $s \in \mathbb{R}^{D(n_0,\ldots,n_{m-1};1)} \setminus N_1$ and let $\Cmi$ be the associated canonical polyhedral complex. For each node index $j \in \{1,\ldots,n_m\}$,  consider the parametrized family 
\[(\mathcal{F}_{m,j})_{s}: \mathbb{R}^{n_0} \times \mathbb{R}^{n_{m-1} + 1} \rightarrow \mathbb{R}\] 
describing the action on $\mathbb{R}^{n_0}$ of neural networks of architecture $(n_0,\ldots,n_{m-1};1)$ whose non-affine-linear layer maps are parameterized by $\tilde{s}$ and whose final affine-linear map to $\mathbb{R}$ is parameterized by the point in $\mathbb{R}^{n_{m-1} + 1}$.
By construction, for any choice of $w \in \mathbb{R}^{n_{m-1} + 1}$, the function  $\left(\mathcal{F}_{m,j}\right)_s \vert_w :\mathbb{R}^{n_0} \to \mathbb{R}$ is 
is affine-linear on the cells of $\Cmi$.
Therefore, by Proposition \ref{prop:ParTransPolyCpx}, to prove that $\left(\mathcal{F}_{m,j}\right)_s \vert_w$ is transverse to $\{0\} \subset \mathbb{R}$ for $\lambda_{n_{m-1}+1}$-almost every $w \in \mathbb{R}^{n_{m-1} + 1}$, it suffices to show that for every cell $C \in \Cmi$,  the parametrized map 
$$(\mathcal{F}_{m,j})_s\vert_{C' \times \mathbb{R}^{n_{m-1}+1}} : C' \times \mathbb{R}^{n_{m-1}+1} \to \mathbb{R}$$
is transverse to $\{0\}$, where $C'=C$ if $\textrm{dim}(C) = 0$ and $C' = \textrm{int}(C)$ otherwise. 
For any such $C$, to show that $(\mathcal{F}_{m,j})_s \vert_{C' \times \mathbb{R}^{n_{m-1}+1}}$ is transverse to $ \{0\} \subseteq \mathbb{R}$, it suffices to show that $(\mathcal{F}_{m,j})_s \vert_{C' \times \mathbb{R}^{n_{m-1}+1}}$ is surjective, since the whole space, $\mathbb{R}$, is clearly transverse to any embedded submanifold. 
But this is clearly true, because for any non-empty cell, $C \in \mathcal{C}(F_{m-1} \circ \ldots \circ F_1)$, there exists a point $p \in F_{m-1} \circ \ldots \circ F_1(C') \subset \mathbb{R}^{n_{m-1}}$, and it is clear that there exists  {\em some} affine-linear transformation $\mathbb{R}^{n_{m-1}} \rightarrow \mathbb{R}$ sending $p$ to $t$.



Thus,  for $\lambda_{n_{m-1}+1}$-almost every $w \in \mathbb{R}^{n_{m-1} + 1}$, the node map $(\mathcal{F}_{m,j})_s \vert_w$ is transverse to $\{0\}$; let $Y_{s,j}$ be the  $\lambda_{n_{m-1}+1}$-null set where this fails. 
Set $Y_s$ to be the set of points $w \in \mathbb{R}^{\delta}$ such that for some $j \in \{1,\ldots,n_m\}$, the projection of $(\tilde{s},w)$ to its $(n_{m-1}+1)$ coordinates representing the affine-linear layer map of the node map $F_{m,j}$ is in 
$Y_{s,j}$.  Tonelli's Theorem implies $\lambda_{\delta}(Y_s) = 0$. 
\end{proof}

\section{Binary codings of regions of co-oriented hyperplane arrangements and the gradient vector field of a ReLU neural network map} \label{sec:Binary}

 In this section, we collect elementary facts about co-oriented hyperplane arrangements that will be useful in the proofs of Theorems \ref{thm:JohnsonResult} and \ref{thm:oneboundedcomponent}. We also introduce a partial orientation on $\mathcal{C}(F)_1$, the $1$--skeleton of the bent hyperplane arrangement of a generic, transversal ReLU neural network, defined using the gradient of the neural network function $F$. This partially-oriented graph plays an important role in the proof of Theorem \ref{thm:oneboundedcomponent}. 
\subsection{Regions, vertices, and edges of classical hyperplane arrangements}

\begin{definition} A {\em region} of a (possibly ordered, co-oriented) hyperplane arrangement $\mathcal{A}$ in $\mathbb{R}^n$ is a connected component of $\mathbb{R}^n \setminus \bigcup_{H \in \mathcal{A}} H$. Let $r(\mathcal{A})$ denote the number of regions of $\mathcal{A}$.
\end{definition}

Note that each region, $R$, of an ordered, co-oriented hyperplane arrangement ${\bf A} = \{{\bf H}_1, \ldots, {\bf H}_k\}$ is naturally labeled with a binary $k$--tuple, $\vec{\theta} \in \{0,1\}^k$, where the $i^{\textrm{th}}$ component of $\vec{\theta}$ associated to $R$ is $1$ (resp. $0$) if the co-orientation of ${\bf H}_i$ points towards (resp., away from) $R$. 
We shall denote the region of ${\bf A}$ labeled by the binary $k$--tuple $\vec{\theta}$ by  $R_{\vec{\theta}}({\bf A})$ and refer to it as the {\em $\vec{\theta}$ region of  ${\bf A}$.} If the ordered, co-oriented hyperplane arrangement ${\bf A}$ is clear from context, we will abbreviate the notation to $R_{\vec{\theta}}$.  We will use $\bar{R}_{\vec{\theta}}({\bf A})$ to denote the closure of $R_{\vec{\theta}}({\bf A})$.

The assignment of binary $k$--tuples to regions of an ordered, co-oriented hyperplane arrangement ${\bf A}$ is clearly injective, but it need not be surjective (see Figure \ref{fig:NotSurjective}). Lemma  \ref{lem:NumRegions} gives a sufficient condition for the assignment to be bijective.

\begin{figure}
 \centering
   \includegraphics[width=0.5\textwidth]{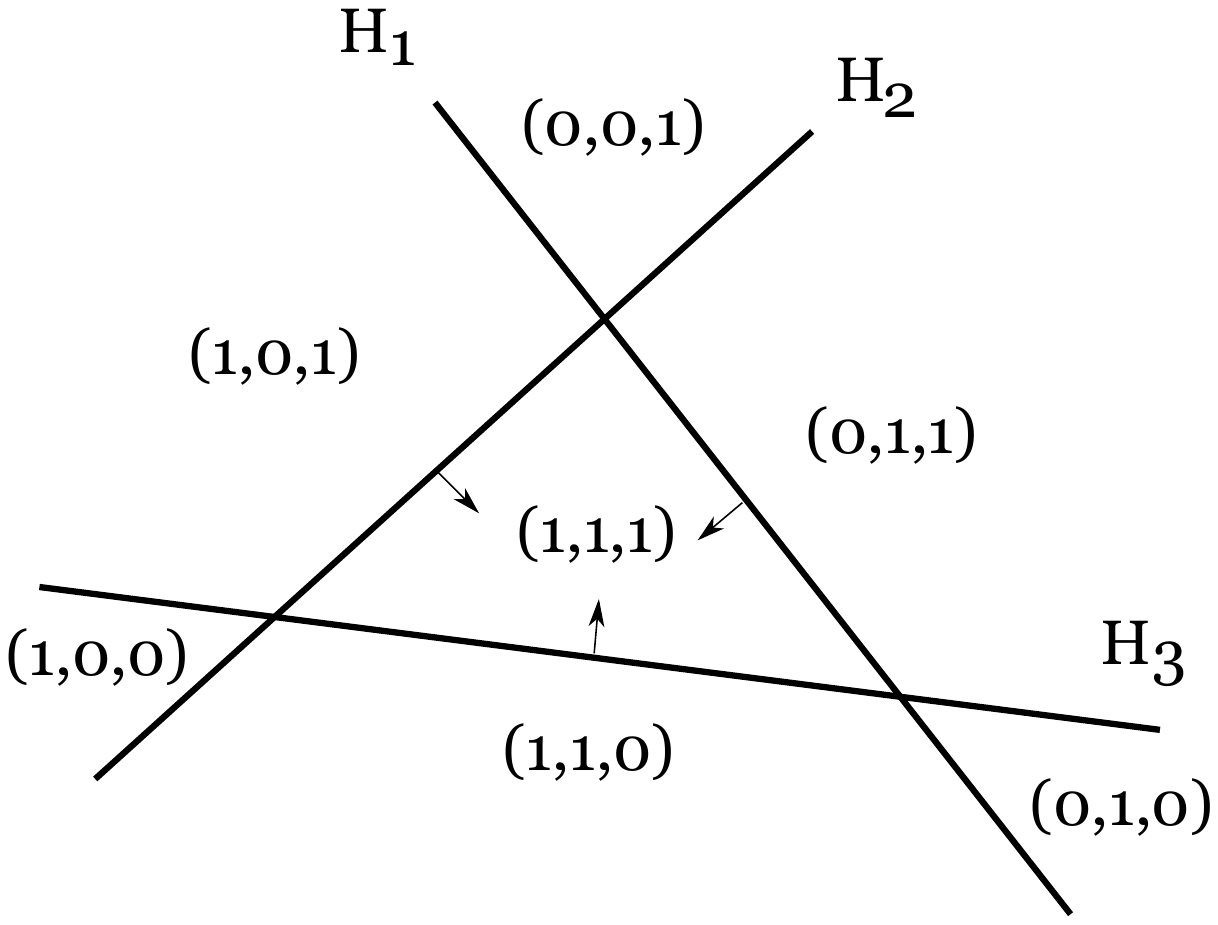}
      \caption{There are $7$ regions in the complement of $3$ co-oriented hyperplanes in $\mathbb{R}^2$. The assignment of binary $3$--tuples to regions is not surjective.}
	\label{fig:NotSurjective}
\end{figure}

The contents of Lemmas \ref{lem:NumRegions} and \ref{lem:klessthann} are well known and follow immediately from a classical theorem of Zaslavsky \cite{Zaslavsky} (cf. \cite[Thm. 2.5]{Stanley}).  

\begin{lemma} \label{lem:NumRegions} Let ${\bf A} := \{{\bf H}_1, \ldots, {\bf H}_k\}$ be a generic ordered, co-oriented hyperplane arrangement in $\mathbb{R}^n$, where $k \leq n$. Then $r({\bf A}) = 2^k$, and hence there is a one-to-one correspondence between regions of ${\bf A}$ and binary $k$--tuples.
\end{lemma}

\begin{lemma} \label{lem:klessthann} Let $\bf{A} = \{H_1, \ldots, H_k\}$ be a (generic or non-generic) ordered, co-oriented hyperplane arrangement of $k \leq n$ hyperplanes in $\mathbb{R}^n$. $\bf{A}$ has no bounded regions.
\end{lemma}

\begin{lemma} \label{lem:vertexint} 
Let $\mathcal{A} = \{H_1, \ldots, H_{N}\}$  be any arrangement of $N$ hyperplanes in $\mathbb{R}^n$.  Then for any rank $n$, size $n$ subarrangement $\{H_{i_1}, \ldots, H_{i_n}\}$ of $\mathcal{A}$, $H_{i_1} \cap \ldots \cap H_{i_n}$ is a $0$--cell (vertex) of the canonical polyhedral complex $\mathcal{C}(\mathcal{A})$.  Conversely, every $0$-cell of $\mathcal{C}(\mathcal{A})$ can be realized as  $H_{i_1} \cap \ldots \cap H_{i_n}$ for some rank $n$, size $n$ subarrangement $\{H_{i_1}, \ldots, H_{i_n}\} \subseteq \mathcal{A}$.
\end{lemma}

\begin{proof} Let $\mathcal{A}' = \{H_{i_1}, \ldots, H_{i_n}\}$ be a rank $n$, size $n$ subarrangement of $\mathcal{A}$. From Lemma \ref{lem:GenericArrangementInvertibleMap} it follows that $\mathcal{A}'$ is a generic arrangement, and hence the $n$--fold intersection, $H_{i_1} \cap \ldots \cap H_{i_n}$, is an affine subspace of $\mathbb{R}^n$ of dimension $0$. Indeed, since the $(n-k)$--dimensional affine subspaces of $\mathbb{R}^n$ associated to $k$--fold intersections of $k$--element subsets of $\mathcal{A}'$ are reverse-ordered by inclusion, we see that $p = H_{i_1} \cap \ldots \cap H_{i_n}$ is the unique $0$--cell in the boundary of all cells of the polyhedral complex $\mathcal{C}(\mathcal{A}')$. Since $\mathcal{A}$ is obtained from $\mathcal{A}'$ by adding hyperplanes, $\mathcal{C}(\mathcal{A})$ is a polyhedral subdivision of $\mathcal{C}(\mathcal{A}')$, and so $p$ is also a $0$--cell (vertex) of $\mathcal{C}(\mathcal{A})$, as desired.

For the converse statement, we proceed by induction on $n$. For the base case ($n =1$), it follows directly from the definition of $\mathcal{C}(\mathcal{A})$ that every $0$--cell (vertex) is a hyperplane of $\mathcal{A}$. Now let $n>1$ and suppose $p$ is a $0$--cell of $\mathcal{C}(\mathcal{A})$. We know that $p \in K$ for some hyperplane $K \in \mathcal{A}$. Consider the restricted {\em solution set arrangement}, $\mathcal{S}_K = \{K \cap H\,\,|\,\, H \in \mathcal{A}\setminus K\}$, from which we obtain a restricted hyperplane arrangement, $\mathcal{A}_K$, by deleting the degenerate solution sets. Then $p$ is also a $0$--cell in the canonical polyhedral complex $\mathcal{C}(\mathcal{A}_K)$. Since $\mathcal{A}_K$ is an $(n-1)$--dimensional hyperplane arrangement, the inductive hypothesis tells us that there exist hyperplanes $H_{i_1}, \ldots, H_{i_{n-1}} \in \mathcal{A}$ such that 
$p = (K \cap H_{i_1}) \cap \ldots \cap (K \cap H_{i_{n-1}})$ in $K$. Letting $H_{i_n} = K$, it follows that $p = H_{i_1} \cap \ldots \cap H_{i_n}.$  Moreover, $\mathcal{A}'$ must be rank $n$, since otherwise its intersection would be an affine space of dimension $>0$.
\end{proof}

\begin{corollary} \label{cor:vertexenum} 

Let $\mathcal{A} = \{H_1, \ldots, H_N\}$  be any arrangement of $N$ hyperplanes in $\mathbb{R}^n$ obtained from the rows,  
$\{(W_i|b_i)\}_{i=1}^N$, of the weight-bias matrix associated to a nondegenerate affine linear transformation as in equation \eqref{eq:affsolnset}. 
 There is a canonical surjective map from the set of linearly-independent $n$--element subsets (subbases), $\{W_{i_1}, \ldots, W_{i_n}\} \subseteq \{W_1, \ldots, W_N\}$, to the set of vertices of 
 $\mathcal{C}(\mathcal{A})$.
\end{corollary}

\begin{proof} Subbases of the set $\{W_1, \ldots, W_N\}$ of weight vectors are in canonical bijective correspondence with the set of rank $n$, size $n$ subarrangements of the arrangement $\mathcal{A} = \{H_1,\ldots,H_n\}$. By Lemma \ref{lem:vertexint}, there is a surjective map from the set of rank $n$, size $n$ subarrangements of $\mathcal{A}$ to the set of vertices of $\mathcal{C}(\mathcal{A})$, defined by taking the $n$--fold intersection of the hyperplanes in the subarrangement.
\end{proof}

\begin{lemma} \label{lem:vertexadj} Let $\mathcal{A} = \{H_1, \ldots, H_N\} \subset \mathbb{R}^n, \{(W_i|b_i)\}_{i=1}^N$ be as above, and let \[\mathcal{B}_p = \{W_{i_1}, \ldots, W_{i_n}\} \mbox{ and } \mathcal{B}_q = \{W_{j_1}, \ldots, W_{j_n}\}\] be two subbases of $\{W_1, \ldots, W_N\}$ with corresponding canonical vertices $p$ and $q$, respectively, as guaranteed by Corollary \ref{cor:vertexenum}. If $|\mathcal{B}_p \cap \mathcal{B}_q| = n-1,$ then either $p=q$ or $p \cup q$ is the boundary of a $1$--cell in $\mathcal{C}(\mathcal{A})$.
\end{lemma}

\begin{proof} If \[\mathcal{B}_p = \{W_{i_1}, \ldots, W_{i_n}\} \mbox{ and } \mathcal{B}_q = \{W_{j_1}, \ldots, W_{j_n}\}\] are as above, then (reordering $j_1, \ldots, j_n$ if necessary) we may assume that $i_k = j_k$ for $k=1, \ldots, n-1$ and $i_n \neq j_n$. Now (the proof of) Corollary \ref{cor:vertexenum} tells us that \[p = \left(H_{i_1} \cap \ldots \cap H_{i_{n-1}}\right) \cap H_{i_n} \mbox{ and } q = \left(H_{i_1} \cap \ldots \cap H_{i_{n-1}}\right) \cap H_{j_n},\] which tells us that $p$ and $q$ are points on the same $1$--dimensional affine space, $\left(H_{i_1} \cap \ldots \cap H_{i_{n-1}}\right) \subseteq \mathcal{C}(\mathcal{A})$. 
\end{proof}

\begin{corollary} \label{cor:vertexadjnplus1} Let $\mathcal{A} = \{H_1, \ldots, H_{n+1}\}$ be any (generic or non-generic) arrangement of hyperplanes in $\mathbb{R}^n$. Every pair of $0$--cells (vertices), $p \neq q$, of $\mathcal{C}(\mathcal{A})$ is connected by some $1$--cell (edge) of $\mathcal{C}(\mathcal{A})$. That is, all vertices are adjacent in the graph $\mathcal{C}(\mathcal{A})_1$.
\end{corollary}

\begin{proof} Every pair of $n$--element subsets of an $(n+1)$--element set has a common $(n-1)$--element subset. So if $\mathcal{B}_p, \mathcal{B}_q$ are two subbases of $\{W_1, \ldots, W_{n+1}\}$ and $p \neq q$, then Lemma \ref{lem:vertexadj} tells us $p$ and $q$ are adjacent in $\mathcal{C}(\mathcal{A})_1$.
\end{proof}

The following lemma is immediate.
 
\begin{lemma} \label{l:stratification}
Let ${\bf A} = \{{\bf H}_1, \ldots, {\bf H}_n\}$  be a generic, ordered, co-oriented arrangement of $n$ hyperplanes in $\mathbb{R}^n$. 
Then the map from $\{0,1\}^n$ to the set of faces of the polyhedral set $\bar{R}_{\vec{1}}$ given by 
\begin{equation} \label{eq:Fthetadef}
\vec{\theta} \in \{0,1\}^n \ \mapsto F_{\vec{\theta}} \ := \{\vec{\theta} \odot \vec{v}\,\,|\,\,\vec{v} \in \bar{R}_{\vec{1}}\}
\end{equation}
is a bijection. 
\end{lemma}

 Here, $\vec{1}$ denotes the vector whose entries are all $1$, and  $\vec{\theta} \odot \vec{v}$ denotes the {\em Hadamard product} (component-wise product) of $\vec{\theta}$ and $\vec{v}$.  Accordingly, $F_{\vec{1}} = \bar{R}_{\vec{1}}$, and $F_{\vec{0}} = \{(0,\ldots, 0)\}$. Moreover, $\mbox{dim}(F_{\vec{\theta}}) = \sum_{i=1}^n \theta_i$.

Let ${\bf A}^{st}$ be the standard ordered, co-oriented coordinate hyperplane arrangement in $\mathbb{R}^n$. That is, ${\bf A}^{st} = \{{\bf H}^{st}_i\}_{i=1}^n$, where \[{\bf H}^{st}_i := \{\vec{v} = (v_1, \ldots, v_n) \in \mathbb{R}^n \,\,|\,\, v_i = 0\},\] co-oriented in the direction of the non-negative half-space, $\{\vec{v} \in \mathbb{R}^n \,\,|\,\, v_i \geq 0\}.$ 
For $\vec{\theta} \in \{0,1\}^n$, we shall denote by $F^{st}_{\vec{\theta}}$ the corresponding face of the polyhedral set $\bar{R}^{st}_{\vec{1}}$, as in \eqref{eq:Fthetadef}. In summary, we have:

\begin{lemma} \label{lem:ReLULayer}
Let ${\bf A} = \{{\bf H}_1, \ldots, {\bf H}_n\}$  be a generic, ordered, co-oriented arrangement of $n$ hyperplanes in $\mathbb{R}^n$ associated to a generic layer map $\sigma \circ A:\mathbb{R}^n \to \mathbb{R}^n$.  
 Then $A$ maps the ${\bf A}$ decomposition of $\mathbb{R}^n$ to the ${\bf A}^{st}$ decomposition of $\mathbb{R}^n$, in the following sense:
\begin{itemize}
	\item $A({\bf H}_i) = {\bf H}_i^{st}$ for all $i=1, \ldots, n$, and
	\item $A(R_{\vec{\theta}}) = R_{\vec{\theta}}^{st}$ and $A(F_{\vec{\theta}}) = F_{\vec{\theta}}^{st}$ for all $\vec{\theta}\in \{0,1\}^n$.
\end{itemize}
Moreover, $\sigma \circ A$ is the composition of the affine isomorphism $A$ realizing the above identification, followed by the projection $\bar{R}_{\vec{\theta}}^{st} \rightarrow F_{\vec{\theta}}^{st}$ given by taking the Hadamard product with $\vec{\theta} \in \{0,1\}^n \subseteq \mathbb{R}^n$.
\end{lemma}

\subsection{ Regions of bent hyperplane arrangements and the $\nabla F$--oriented $1$--skeleton of the canonical polyhedral complex}

We can  similarly endow the activation regions of a generic, transversal ReLU neural network with a binary labeling as follows.  Recall that if $F: \mathbb{R}^{n_0} \rightarrow \mathbb{R}$ 
is transversal and generic, Theorem \ref{thm:nminus1skel} guarantees that the domain of the $(n_0-1)$--skeleton of $\mathcal{C}(F)$ (resp., the $n_0$--cells) agrees with the bent hyperplane arrangement, $\mathcal{B}_F$, (resp., the closures of the activation regions of $\mathcal{B}_F$). 

In this case, the image of every activation region of $F$ (interior of an $n_0$--cell of $\mathcal{C}(F)$) is contained in a unique region of the co-oriented hyperplane arrangement in each layer: 

\begin{definition} \label{def:sequenceOfBinaryTuples}
Let 
\[F: \xymatrix@1{\mathbb{R}^{n_0} \ar[r]^{F_1} & \mathbb{R}^{n_1} \ar[r]^{F_2} & \ldots \ar[r]^{F_m} & \mathbb{R}^{n_m} \ar[r]^G & \mathbb{R}}\] 
be a transversal, generic ReLU neural network, and let ${\bf A}^{(i)}$ denote the co-oriented hyperplane arrangement associated to $F_i$. 
The \emph{$(\vec{\theta}_1, \ldots, \vec{\theta}_m)$--region} of $F$, denoted $R_{(\vec{\theta}_1, \ldots, \vec{\theta}_m)}$, is the unique activation region of $F$ satisfying the property that for each $i \in \{1, \ldots m\}$, 
\begin{equation*} \label{eq:binarytuplesequence}
F_{i-1} \circ \ldots \circ F_1 \left(R_{(\vec{\theta}_1, \ldots, \vec{\theta}_m)}\right) \subseteq R_{\vec{\theta}_{i}}\left({\bf A}^{(i)}\right). 
\end{equation*}
\end{definition}


\begin{lemma} \label{lem:compgradFregion} Let $$F: \mathbb{R}^{n_0}  \xrightarrow{F_1} \mathbb{R}^{n_1}  \xrightarrow{F_2} \mathbb{R}^{n_2}  \xrightarrow{F_3}\dots \xrightarrow{F_m} \mathbb{R}^{n_m}  \xrightarrow{G} \mathbb{R}^1$$ be a generic, transversal ReLU neural network with associated  weight  matrices $W_1, \ldots, W_m, W_{m+1}$.  
Let $\bf{p}$ be a point in a region $R_{(\vec{\theta}_1, \ldots, \vec{\theta}_m)}$ with associated  sequence of binary tuples $(\vec{\theta}_1, \ldots, \vec{\theta}_m)$.
 Then \[\left.\frac{\partial F}{\partial {\bf x}}\right|_{{\bf x} = {\bf p}} = W_{m+1}W^{\vec{\theta}_m}_m \cdots W^{\vec{\theta}_1}_1,\] where we define $W^{\vec{\theta}_k}_k$ to be the matrix obtained from $W_k$ by replacing the $i$th row of $W_k$ with $0$'s when the $i$th entry of $\vec{\theta}_k$ is $0$.
\end{lemma}

\begin{proof} Immediate from the definition of the affine-linear function on $R_{(\vec{\theta}_1, \ldots, \vec{\theta}_m)}$ and the chain rule for partial derivatives. 
\end{proof}

The $1$-skeleton, $\mathcal{C}(F)_1$, of the canonical polyhedral decomposition for a neural network $F:\mathbb{R}^{n_0} \to \mathbb{R}$, is an embedded linear graph in $\mathbb{R}^{n_0}$.  This graph has a natural partial orientation, defined as follows.

\begin{definition} \label{defn:oriented1skel}
Let $F:\mathbb{R}^{n_0} \to \mathbb{R}$ be a neural network, denote by $\mathcal{C}(F)_1$ the $1$--skeleton of its canonical polyhedral complex, and let $C$ be a $1$--cell in $\mathcal{C}(F)_1$.  
\begin{itemize}
	\item If $F$ is nonconstant on $C$, orient $C$ in the direction in which $F$ increases. 
	\item If $F$ is constant on $C$, we will leave it unlabeled and refer to $C$ as a {\em flat} edge.
\end{itemize}
We will refer to this (partial) orientation on $\mathcal{C}(F)_1$ as the $\nabla F$--orientation.
\end{definition}

\noindent Note that $F$ is affine-linear on $C$ by Theorem~\ref{thm:BHAispolycpx}. 

\section{Obstructions to Topological Expressiveness and Applications} \label{sec:TopApp}

This section uses the framework developed in the previous sections to give an alternative perspective on the Beise-Cruz-Schr{\"o}der, Hanin-Sellke, Johnson result that a width $n$ ReLU neural network  $F: \mathbb{R}^n \rightarrow \mathbb{R}$ has decision regions that are either empty or unbounded. We also develop an architecture-based obstruction to the existence of multiple bounded connected components in a decision region.

Recall the statement of the Beise-Cruz-Schr{\"o}der, Hanin-Sellke, Johnson result:

\begin{theoremJohnsonResult} \cite{DDCruz, HaninSellke, Johnson}
 For any integer $n \geq 2$,  let $F: \mathbb{R}^n \rightarrow \mathbb{R}$ be a ReLU neural network, all of whose hidden layers have dimension $\leq n$.  
Then for any decision threshold $t$, each of the sets $N_F(t), B_F(t), Y_F(t)$ (defined in \eqref{decisionregionsnotation}) 
is either empty or unbounded.
\end{theoremJohnsonResult}

We will need some elementary facts about the image of a ReLU neural network in the width $n$ case. We address the case of generic and non-generic layer maps separately. 

\begin{proposition} \label{prop:ImN} Let \[N = (F_m \circ \ldots \circ F_1): \mathbb{R}^n \rightarrow \ldots \rightarrow \mathbb{R}^n\] be the composition of all but the final layer map of a generic width $n$ ReLU neural network in which every hidden layer has dimension $n$.  Then $\mbox{Im}(N) \subseteq \bar{R}_{\vec{1}}^{st} \subseteq \mathbb{R}^n$ is the domain of a polyhedral complex, $\mathcal{C}$, with at most one cell of dimension $n$. Explicitly, \[\mbox{Im}(N) = |\mathcal{C}| = \mathcal{P} \cup \mathcal{Q},\] where $\mathcal{P}$ is a (possibly empty) polyhedral set of dimension $n$, and $\mathcal{Q}$ is a union of polyhedral sets of dimension $< n$. 
Moreover, if $\mathcal{P}$ is nonempty, $n \geq 2$, and
 $\mathcal{P} = H_1^+ \cap \ldots \cap H_m^+$
 is an irredundant realization of $\mathcal{P}$ as an intersection of closed half spaces,  then 
 $\{H_1,\ldots,H_m\}$ has rank $\geq 2$. 

\end{proposition}

\begin{proof} We proceed by induction on the number, $m$, of layers. If $m=1$, Im($N$) is $\bar{R}_{\vec{1}}^{st}$, which is the domain of a polyhedral complex with a single $n$--cell, $\mathcal{P} = \bar{R}_{\vec{1}}^{st}$, realizable as an irredundant intersection of $n$ half-spaces. Moreover,  if $n \geq 2$, the rank of the corresponding hyperplane arrangement, $\{H^{st}_1, \ldots, H^{st}_n\}$ is $\geq 2$, as required.

Now consider $m \geq 2$ and suppose $\mathcal{C}'$ is the polyhedral complex whose domain agrees with the image, $\mbox{Im}(N') = \mathcal{P}' \cup \mathcal{Q}'$, of the first $m-1$ layers of $N$ as described.
 Let 
\begin{itemize}
	\item ${\bf A}$ be the ordered, co-oriented hyperplane arrangement associated to $F_m$, 
	\item $\mathcal{C}({\bf A})$ be the associated polyhedral decomposition of the domain of $F_m$,
	\item and $\mathcal{C}'_{\cap \mathcal{C}({\bf A})}$ be the complex whose cells are pair-wise intersections of cells of $\mathcal{C}'$ with cells of $\mathcal{C}({\bf A})$. 
\end{itemize}

\noindent By Lemma \ref{lem:ReLULayer}, $F_m$ is affine-linear on the cells of $\mathcal{C}'_{\cap \mathcal{C}({\bf A})}$. Indeed, we can alternatively characterize $\mathcal{C}'_{\cap \mathcal{C}({\bf A})}$ as the level set complex, $\mathcal{C}'_{\in \mathcal{C}\left({\bf A}^{st}\right)}$, of $F_m$ relative to the polyhedral decomposition associated to the standard hyperplane arrangement, ${\bf A}^{st}$, in the codomain. 

Noting that the image of a polyhedral set under an affine-linear map is a polyhedral set, we now define $\mathcal{C}$ to be the complex whose cells are the images of cells of $\mathcal{C}'_{\in \mathcal{C}\left({\bf A}^{st}\right)}$ under $F_m$.

It then follows immediately from Lemma \ref{lem:ReLULayer} that all cells of $\mathcal{C}$ have dimension $<n$ except possibly $F_m\left(\mathcal{P}' \cap \bar{R}_{\vec{1}}({\bf A})\right)$. Further, since $F_m$ is an affine isomorphism on $\bar{R}_{\vec{1}}({\bf A})$, this cell will be $n$--dimensional iff $\mathcal{P}' \cap \bar{R}_{\vec{1}}({\bf A})$ is $n$--dimensional. 

In this case,  we claim that as long as $n \geq 2$, an irredundant bounding hyperplane arrangement of the cell $\mathcal{P}' \cap \bar{R}_{\vec{1}}({\bf A})$ will have rank $\geq 2$, 
since both $\mathcal{P}'$ and $\bar{R}_{\vec{1}}({\bf A})$ have this property. To see this, note that the union of the bounding hyperplane arrangements for $\mathcal{P}'$ and $\bar{R}_{\vec{1}}({\bf A})$ yields a bounding hyperplane arrangement for $\mathcal{P}' \cap \bar{R}_{\vec{1}}({\bf A})$, and it necessarily has rank $\geq 2$. If this union is irredundant, we are done. If not, we appeal to Farkas' Lemma III, cf. \cite[Sec. 2]{Ziegler}, which says that any {\em redundant} inequality in a system of linear inequalities is a non-negative linear combination of the other linear inequalities in the system. This implies that the rank of any bounding hyperplane arrangement of a polyhedral set is equal to the rank of an {\em irredundant} bounding hyperplane arrangement.

Moreover, the image under $F_m$ of an irredundant bounding hyperplane arrangement for  $\mathcal{P}' \cap \bar{R}_{\vec{1}}({\bf A})$ is also irredundant and has rank $\geq 2$, 
 since $F_m$ is an affine isomorphism on $\bar{R}_{\vec{1}}({\bf A})$.  Defining $\mathcal{P}$ to be 

\begin{itemize}
	\item $F_m\left(\mathcal{P}'_{\in \bar{R}_{\vec{1}}({\bf A}^{st})}\right)$ if the polyhedral set $\mathcal{P}'_{\in \bar{R}_{\vec{1}}({\bf A}^{st})}$ has dimension $n$, and
	\item $\emptyset$ otherwise,
\end{itemize}

and $\mathcal{Q}$ to be the union of all other cells of $\mathcal{C}$, the result follows.

\end{proof}

\begin{lemma} \label{lem:BdryPreImUnbound} Let $N = (F_m \circ \ldots \circ F_1): \mathbb{R}^n \rightarrow  \mathbb{R}^n$ be the composition of all but the final layer map of a generic width $n$ ReLU neural network, and let 
$\textrm{Im}(N) = |\mathcal{C}| = \mathcal{P} \cup \mathcal{Q}$ as in Proposition \ref{prop:ImN} above. If  $x \in |\mathcal{C}_{n-1}|$,  
then $N^{-1}(\{x\})$ is unbounded.
\end{lemma}

\begin{proof}  We proceed  by induction on the number, $m$, of layers. The result is clear when $m=1$, since in this case Lemma \ref{lem:ReLULayer} tells us that $\mathcal{Q} = \emptyset$ and $\mathcal{P} = \bar{R}_{\vec{1}}^{st}$, and each point $x \in |\mathcal{C}_{n-1}| = \partial \mathcal{P} = \partial(\bar{R}_{\vec{1}}^{st})$ is in the image of the projection map $R^{st}_{\vec{\theta}} \rightarrow F^{st}_{\vec{\theta}}$, hence has unbounded preimage. 

 Now suppose $m > 1$ and the result holds for the image, $|\mathcal{C}'| = \mathcal{P}' \cup \mathcal{Q}'$, of the first $m-1$ layers of $N$. That is, each point in $|\mathcal{C}'_{n-1}|$ has unbounded preimage. Let $x \in |\mathcal{C}_{n-1}| \subseteq \mbox{Im}(N)$. If $x$ is in the image of $|\mathcal{C}'_{n-1}|$, then it has unbounded preimage by the inductive hypothesis.  So we may assume that $x$ is in the image of $\mbox{int}(\mathcal{P}')$ and not in the image of $\partial\mathcal{P}'$. But since $x \in |\mathcal{C}_{n-1}|$, Lemma \ref{lem:ReLULayer} implies that 
$F_m^{-1}(\{x\}) \subseteq |\mathcal{C}'|$ is a ray contained in $\mbox{int}(\mathcal{P}')$. Since $\mbox{int}(\mathcal{P}')$ is the image of the interior of a polyhedral set in the domain under a composition of affine-linear isomorphisms, the preimage of this ray is a ray in the domain, hence unbounded. The conclusion follows. 
\end{proof}

\begin{lemma} \label{lem:HypIntBdry}
 Let $\mathcal{P} =  H_1^+ \cap \ldots \cap H_m^+ \subseteq \mathbb{R}^n$ be an irredundant representation of a non-empty $n$--dimensional polyhedral set such that the hyperplane arrangement $\mathcal{A} = \{H_1, \ldots, H_m\}$ has rank $\geq 2$, and let $F_i = \mathcal{P} \cap H_i$ be the bounding facet in $\partial \mathcal{P}$ corresponding to $H_i$.  If $X \subseteq \mathbb{R}^n$ is any affine hyperplane satisfying $X \cap \mathcal{P} \neq \emptyset$, then $X \cap F_i \neq \emptyset$ for some $i$.
\end{lemma}

\begin{proof} Let $X$ and $\mathcal{P}$ in $\mathbb{R}^n$ be as described above. Note that the rank assumption implies that $m, n \geq 2$.  Since $X \cap \mathcal{P} \neq \emptyset$, then $X \subseteq \mathcal{P}$ or $X \cap \partial\mathcal{P} \neq \emptyset$. But if $X \subseteq \mathcal{P}$, then $X$ must be parallel to every $H_i$, contradicting the rank assumption. So $X \cap \partial\mathcal{P} \neq \emptyset$. But $\partial\mathcal{P} = F_1 \cup \ldots \cup F_m$, so $X \cap F_i \neq \emptyset$ for some $i$, as desired.
\end{proof}

\begin{lemma} \label{lem:NonGeneric} For any $n \in \mathbb{N}$, let $N = (F_m \circ \ldots \circ F_1): \mathbb{R}^n \rightarrow \mathbb{R}^n$ be the composition of all but the final layer map of a {\em non}-generic width $n$ ReLU neural network. Then each point in $Im(N)$ has unbounded preimage. 
\end{lemma}

\begin{proof} Let $F_i$ be the first non-generic layer map in $N$. Since $F_i$ is non-generic, the affine map $A_i$ underlying $F_i$ is non-invertible, by Lemma \ref{lem:GenericArrangementInvertibleMap}. Indeed, the preimage of any point is an affine-linear subspace of $\mathbb{R}^n$ of dimension $\geq 1$, hence unbounded. Recalling that $F_i = \sigma \circ A_i$ is a  map that is affine-linear on the cells of the canonical polyhedral decomposition  $\mathcal{C}(F_{i-1} \circ \ldots \circ F_i)$
of $\mathbb{R}^n$, it follows immediately that the affine-linear map on each cell is also non-invertible, hence the preimage of {\em any} point $p \in \mbox{Im}(F_i)$ is unbounded. Any point $q \in \mbox{Im}(N)$ is of the form $q = (F_m \circ \ldots \circ F_{i+1})(p)$ for some $p \in \mbox{Im}(F_i)$. So $N^{-1}(\{q\})$ is unbounded.
\end{proof}

Recall that a threshold $t$ is transversal for a neural network $F$  (Definition \ref{def:NeuralNetworkTransversalThreshold}) if it is transversal for $F$ with respect to its canonical polyhedral complex $\mathcal{C}(F)$.

\begin{lemma} \label{lem:TransThreshDecisionBoundary}
Let $t$ be a transversal threshold for a neural network $F:\mathbb{R}^{n_0} \to \mathbb{R}$.  Then 
$$B_F(t) = \partial N_F(t) = \partial Y_F(t).$$
\end{lemma}

\begin{proof}
In the case that $B_F(t) = \emptyset$, either $Y_F(t) = \mathbb{R}^{n_0}$ and $N_F(t) = \emptyset$ or $Y_F(t) = \emptyset$ and $N_F(t) = \mathbb{R}^{n_0}$, so the statement holds. 

Now suppose $B_F(t) \neq \emptyset$.  For each cell $C \in \mathcal{C}(F)$ such that $F^{-1}(\{t\}) \cap C \neq \emptyset$,  Lemma \ref{lem:genthreshconvpoly} guarantees that  $\textrm{aff}(F^{-1}(\{t\}) \cap C)$ is a hyperplane that cuts $C$.  Denote by $C^+$ and $C^-$ the intersections of $C$ with the two open half-spaces that are the complement of this hyperplane. 
 Since $F$ is nonconstant on $C$ (by Lemma  \ref{lem:nonconscellneighborhood}), precisely one of $C^+,C^-$ must be contained in $Y_F(t)$ and the other must be contained in $N_F(t)$.  Therefore $B_F(t) \subseteq \partial Y_F(t)$ and $B_F(t) \subseteq \partial N_F(t)$.  The reverse inclusions are obvious.  
\end{proof}

We are now ready for:

\begin{proof}[Proof of Theorem \ref{thm:JohnsonResult}] 
Fix an integer $n \geq 2$ and let $F: \mathbb{R}^n \rightarrow \mathbb{R}$ be a ReLU neural network whose hidden layers all have dimension $\leq n$.  We may then assume WLOG that every intermediate layer has dimension $n$ (padding points in layers of dimension $<n$ with $0$s as needed).  
Decompose $F$ as $F=G \circ N$, where $N: \mathbb{R}^n \rightarrow \mathbb{R}^n$ is the composition of all the layer maps except for the final one, and $G: \mathbb{R}^n \rightarrow \mathbb{R}$ is the final layer map.   If $G$ is degenerate, it is immediate that the theorem holds.  So assume that $G$ is nondegenerate.  

Step 1: We will prove that for every $t \in \mathbb{R}$, the decision boundary $B_F(t)$ is either empty or unbounded.  Since $G$ is nondegenerate, $X_t := G^{-1}(\{t\})$ is an affine hyperplane in the final hidden layer of $F$. Note that
 \[B_F(t) = N^{-1}\left(G^{-1}\{t\}\right) = N^{-1}\left(\mbox{Im}(N) \cap X_t\right).\]
If $\mbox{Im}(N) \cap X_t$ is empty, then $B_F(t)$ is empty, as desired. So assume $\mbox{Im}(N) \cap X_t$ is nonempty.

Case 1: We first consider the case that  $N$ is non-generic. In this case, $\mbox{Im}(N) \cap X_t \neq \emptyset$ implies $N^{-1}(\mbox{Im}(N) \cap X_t)$ is unbounded by Lemma  \ref{lem:NonGeneric}.

Case 2: Now consider the case that $N$ is generic.  By Proposition \ref{prop:ImN}, $\mbox{Im}(N)$ is the domain of a polyhedral complex $\mathcal{C}$ that has a unique (possibly empty) $n$-cell $\mathcal{P}$.

Subcase a:  If $X_t \cap \mathcal{P} = \emptyset$, then the assumption that $\mbox{Im}(N) \cap X_t \neq \emptyset$ implies  $X_t \cap \mathcal{C}_{n-1} \neq \emptyset$. Therefore $B_F(t)$ is unbounded, by Lemma  \ref{lem:BdryPreImUnbound}. 

Subcase b: If $X_t \cap \mathcal{P} \neq \emptyset$, then $\mathcal{P}$ is nonempty.  Because $n \geq 2$,
 Proposition \ref{prop:ImN} guarantees that $\mathcal{P}$ has rank $\geq 2$. Hence $X_t \cap \mathcal{C}_{n-1} \neq \emptyset$ by Lemma \ref{lem:HypIntBdry}.   Therefore $B_F(t)$ is unbounded by Lemma  \ref{lem:BdryPreImUnbound}. 
 
\medskip
Step 2: We will use the fact that $B_F(t)$ is empty or unbounded to show $N_F(t)$ and $Y_F(t)$ are also. 
 
Case 1: When $t \in \mathbb{R}$ is a transversal  threshold, it is now straightforward to see that the decision regions $Y_F(t)$ and $N_F(t)$ are also either empty or unbounded, since $B_F(t) = \partial N_F(t) = \partial Y_F(t)$ by Lemma 
\ref{lem:TransThreshDecisionBoundary}, and a bounded set cannot have unbounded closure.

Case 2: Suppose $t \in \mathbb{R}$ is a non-transversal  threshold.  We will give an argument for $Y_F(t)$; the argument for $N_F(t)$ is analogous. 
Let $X_t^+$ be the positive half-space associated to the co-oriented affine hyperplane $X_t$. Then 
$$F^{-1}((t, \infty)) = N^{-1}(\mbox{Im}(N) \cap X_t^+).$$ If this intersection is empty, then $Y_F(t)$ is empty, as desired.  If this intersection is nonempty, there are two subcases.

Subcase a: $t < \max\{t' \in \mathbb{R}: F^{-1}((t', \infty)) \neq \emptyset\}$.  In this case, since $F$ is continuous, there exists $\epsilon > 0$ such that 
$$(t,t+\epsilon) \subset \{t' \in \mathbb{R}: F^{-1}((t', \infty)) \neq \emptyset\}.$$
 Hence Lemma \ref{lem:transversalthreshold} implies the existence of a transversal threshold $t' > t$ for which $F^{-1}(\{t'\})$ is non-empty. Noting that $F^{-1}(\{t'\})= B_F(t')$, it follows from the first part of the proof that $B_F(t')$ is unbounded. So $Y_F(t) \supseteq B_F(t')$ must also be unbounded.

Subcase b: $t =  \max\{t' \in \mathbb{R}: F^{-1}((t', \infty)) \neq \emptyset\}$. In this case $Y_F(t)$ is empty, as desired.  
\end{proof}

\subsection{Obstructing multiple bounded connected components}

As observed in \cite{Johnson}, it is straightforward to construct, for every $n$, a width $n+1$ neural network with a single hidden layer, $\mathbb{R}^n \rightarrow \mathbb{R}^{n+1} \rightarrow \mathbb{R}$, that has a bounded decision region consisting of 
 a single connected component. We prove that such a simple architecture cannot produce a decision region with more than one bounded connected component.

\begin{theoremoneboundedcomponent} Let $F: \mathbb{R}^n \rightarrow \mathbb{R}^{n+1} \rightarrow \mathbb{R}$ be a ReLU neural network.
Then a decision region $Y_F(t)$ or $N_F(t)$ associated to a transversal threshold $t$ can have no more than one bounded connected component.
\end{theoremoneboundedcomponent}

The following Lemma is a standard result in linear programming.  

\begin{lemma} \label{lem:Maxminkface} 
Let $\mathcal{P} \subseteq \mathbb{R}^n$ be a polyhedral set, and let $F: \mathcal{P} \rightarrow \mathbb{R}$ be an affine-linear map on $\mathcal{P}$. If $F$ achieves a maximum (resp., minimum) on the interior of $\mathcal{P}$, then $F$ achieves this maximum (resp., minimum) value on all of $\mathcal{P}$, and hence on all faces in its boundary. 
\end{lemma}


\begin{corollary} \label{cor:Maxminkface} 
Let $F:\mathbb{R}^{n_0} \to \mathbb{R}$ be a neural network, and $\mathcal{C}(F)$ be its canonical polyhedral complex. If $\mathcal{P}$ is a cell of $\mathcal{C}(F)$ with at least one vertex as a face, and $F$ achieves a maximum (resp., minimum) on $\mathcal{P}$, then $F$ achieves a maximum (resp., minimum) at a vertex of $\mathcal{P}$.
\end{corollary}

\begin{proof} The function $F$ is affine-linear on $\mathcal{P}$ by Theorem \ref{thm:BHAispolycpx}. Under the given assumptions, every face of $\mathcal{P}$ will also be a polyhedral set with a vertex as a face. The result follows by strong induction on the dimension of $\mathcal{P}$, applying Lemma \ref{lem:Maxminkface} in the inductive step.
\end{proof}

\begin{corollary} \label{cor:Maxminpolytope} Let $F$ and $\mathcal{C}(F)$ be as above. If $\mathcal{P}$ is a bounded cell (polytope) of $\mathcal{C}(F)$ {\em of any dimension}, then $F$ achieves a maximum (resp., minimum) at a vertex of $\mathcal{P}$.
\end{corollary}

\begin{proof} Every polytope $\mathcal{P} \subset \mathbb{R}^n$ (of any dimension) has at least one vertex. Moreover, it is bounded, hence compact, since cells are closed. The extreme value theorem then guarantees that $F$ achieves both a minimum and maximum value on $\mathcal{P}$. The result follows from Corollary \ref{cor:Maxminkface}.
\end{proof} 

\begin{proposition} \label{prop:subgraph}  
Let $t$ be a transversal threshold for a neural network $F: \mathbb{R}^{n_0} \to \mathbb{R}$, and let $S$ be a bounded connected component of $Y_F(t)$ (resp., $N_F(t)$).
Then there exist non-empty bounded subgraphs $\mathcal{G}' \subseteq \mathcal{G} \subseteq \mathcal{C}(F)_1$ which, when  endowed with the $\nabla F$--orientation (Definition \ref{defn:oriented1skel}), satisfy:
\begin{enumerate}
	\item $\mathcal{G}'$ is flat;
	\item $\mathcal{G}' \subseteq \mathcal{G} \subsetneq S$;
	\item there is a non-empty collection, $\mathcal{E}$, of edges adjacent to $\mathcal{G}$, satisfying the property that every edge $e \in \mathcal{E}$ points towards $\mathcal{G}$ (resp., points away from $\mathcal{G}$) and has nonempty intersection with $\partial S$ and the other decision region, $N_F(t)$ (resp, $Y_F(t)$).
\end{enumerate}
\end{proposition}


 \begin{remark} One can view the graph $\mathcal{G}'$ described in Proposition \ref{prop:subgraph} as a piecewise linear analogue of a Morse critical point of index $n$ (resp., $0$). 
 \end{remark}

\begin{proof} Let $\overline{S}$ denote the closure of $S$. Since $\overline{S}$ is closed and bounded, hence compact, the extreme value theorem tells us that $F$ attains its maximum (resp., minimum) value, $M \in \mathbb{R}$ (resp., $m \in \mathbb{R}$), on $\overline{S}$. I.e., there exists $x \in \overline{S}$ such that $F(x) = M$ and $F(y) \leq M$ for all $y \in \overline{S}$. But Lemma \ref{lem:Maxminkface} implies that $F^{-1}(\{M\})$ contains a non-empty subgraph, $\mathcal{G}'$, of  $\mathcal{C}(F)_1$, since a maximum value, if attained on the interior of a cell,  is attained on the whole cell, including its boundary.  

Moreover, $\mathcal{G}' \subset S$, for if $\mathcal{G}' \cap (\partial S \subseteq F^{-1}\{t\}) \neq \emptyset$ then $t = M$, which would imply that $t$ is {\em not} a transversal threshold since its preimage contains a vertex which by definition cannot have a nonconstant cellular neighborhood. 

Let $\mathcal{G}$ be the maximal subgraph of $\mathcal{C}(F)_1 \cap S$ that both contains $\mathcal{G}'$ and is entirely contained in $S$. Properties (i) and (ii) are immediate by construction. To see Property (iii), note that that $\mathcal{C}(F)_1$ is connected and unbounded, so it follows that $\mathcal{C}(F)_1 \cap \partial S \neq \emptyset.$  Since $t$ is a transversal threshold, all points in $\partial \overline{S}$ have nonconstant cellular neighborhood, hence all edges of $\mathcal{C}(F)_1$ intersecting $\partial S$ are oriented, and the orientations are toward (resp., away from) $\mathcal{G}$ if $S \subseteq Y_F(t)$ (resp., $S \subseteq N_F(t)$).
\end{proof}

\begin{proof}[Proof of Theorem \ref{thm:oneboundedcomponent}] 
Since $F$ has a single hidden layer, $\mathcal{C}(F)=\mathcal{C}(\mathcal{A})$, for the hyperplane arrangement, $\mathcal{A} \subseteq \mathbb{R}^n$ associated to the first layer map. We may assume without loss of generality that $|\mathcal{A}| = n+1$, for if the first layer map is degenerate then the neural network has width $n$, and hence its decision regions have no bounded connected components, by Theorem \ref{thm:JohnsonResult}.

Now let $t \in \mathbb{R}$ be a transversal  threshold for a (not necessarily generic or transversal) ReLU network, $F: \mathbb{R}^n \rightarrow \mathbb{R}^{n+1} \rightarrow \mathbb{R}.$ We will show that $Y_F(t)$ has no more than one bounded connected component. The argument for $N_F(t)$ is analogous.

Assume, aiming for a contradiction, that $Y_F(t)$ has more than one bounded connected component. Choose two of these: $S_1$ and $S_2$. As described in Proposition \ref{prop:subgraph}, there exist non-empty bounded subgraphs $\mathcal{G}_i \subset S_i$ of the $1$--skeleton of $\mathcal{C}(A)= \mathcal{C}(F)$ and associated non-empty collections, $\mathcal{E}_i$, of edges adjacent to $\mathcal{G}_i$, equipped with $\nabla F$--orientation pointing towards $\mathcal{G}_i$. For each $S_i$, choose an {\em external} vertex $p_i \subset \mathcal{G}_i$. That is, choose a vertex $p_i \subset \mathcal{G}_i$ in the boundary of an edge of $\mathcal{E}_i$. 

By Corollary \ref{cor:vertexadjnplus1}, $p_1$ and $p_2$ are connected by 
a  $1$--cell  $e \in \mathcal{C}(A)$.   
It follows that $e$ is in both $\mathcal{E}_1$ and $\mathcal{E}_2$. But this is impossible, since it would require $e$ to be oriented in two different directions at once. We conclude that one of $\mathcal{E}_1$ or $\mathcal{E}_2$ must be empty, hence Proposition \ref{prop:subgraph} tells us that one of $S_1, S_2$ must be empty. The result follows.

\end{proof}

\bibliography{ReLUPaper}

\end{document}